\documentclass[12pt]{amsart}

\usepackage{amsfonts, amstext, amsmath, amsthm, amscd, amssymb, enumerate, url}
\usepackage{tikz-cd}
\usepackage{xcolor}
\usepackage{svg}
\usepackage{subfigure}
\usepackage{faktor}
\usepackage{float}
\usepackage{stmaryrd}

\usepackage[parfill]{parskip}

\usepackage[colorlinks,citecolor=cyan,linkcolor=magenta]{hyperref}
\usepackage[nameinlink]{cleveref}
\usepackage{cancel}

\usepackage[margin=1in]{geometry}

\newtheorem{thm}{Theorem}[section]
\newtheorem{cor}[thm]{Corollary}
\newtheorem{lemma}[thm]{Lemma}
\newtheorem{prop}[thm]{Proposition}

\newtheorem*{thm:mainthm}{\Cref{thm:mainthm}}

\theoremstyle{definition}
\newtheorem{defn}[thm]{Definition}

\newtheorem{qn}[thm]{Question}

\numberwithin{equation}{section}
\numberwithin{figure}{section}

\newcommand{\ZZ}{\mathbb{Z}}
\renewcommand{\epsilon}{\varepsilon}
\renewcommand{\phi}{\varphi}

\newcommand{\red}{\mathrm{red}}
\newcommand{\ca}{^{\textrm{cut}}}

\title[Irreducible train tracks for pseudo-Anosov maps]{Irreducible train tracks for pseudo-Anosov homeomorphisms}
\author{Ross Griebenow}
\date{\today}

\begin{document}

\begin{abstract}
    We describe a construction of invariant train tracks with irreducible transition matrix for pseudo-Anosov homeomorphisms. 
    This fills what seems to be a gap in the literature concerning the existence of such train tracks.
    The construction starts with an invariant train track associated to the veering triangulation of the mapping torus of the homeomorphism and then uses the veering property to characterize  branches which obstruct irreducibility, finally modifying the track to bypass these obstructions.
\end{abstract}

\maketitle
 
\section{Introduction}\label{intro}
Given a pseudo-Anosov homeomorphism $f$, an invariant train track provides a combinatorial way of recording how curves and other objects behave under iteration of $f$.
The existence of invariant train tracks with irreducible transition matrix
allows the application of Perron-Frobenius theory to the study of the properties of these maps. 
For example, when an invariant train track is irreducible it can be used to build a Markov partition for the map. 
The existence of irreducible invariant train tracks is often taken for granted 
(See e.g. \cite{FR}, \cite{F}, \cite{J}, \cite{EF}, \cite{M}, \cite{Pe}) 
but to our knowledge no complete proof is currently in the literature. See \Cref{discuss} for more details.

\par The veering triangulation associated to $f$ provides a unified way to study both the topology of its mapping torus and the dynamics of its invariant train tracks.
One important datum attached to a veering triangulation is its \textit{flow graph}, which encodes a Markov partition of a pseudo-Anosov map or flow. Flow graphs have been used by Landry--Minsky--Taylor used to count periodic points of pseudo-Anosov maps, and to define polynomial invariants of pseudo-Anosov flows without perfect fits \cite{LMT1}.
When the flow graph associated to the veering triangulation of a pseudo-Anosov mapping torus is strongly connected, the triangulation can be used to obtain invariant train tracks with 
irreducible transition matrices. 
\par For general veering triangulations, Agol--Tsang have characterized the strongly connected components of the flow graph, which consist of one ``reduced component" from which there are paths to every vertex, and some number of ``infinitesimal cycles" which do not have any paths back to the rest of the graph \cite{AT}. 
By studying the interaction between these infinitesimal components and the invariant train tracks arising from the triangulation, we are able to produce new train tracks where the obstructions to irreducibility coming from the infinitesimal components have been bypassed. This allows us to prove the following:

\begin{thm}\label{mainthm}
    Every pseudo-Anosov homeomorphism has an invariant train track whose transition matrix is irreducible.
\end{thm}

\par In slightly more detail, 
given a boundaryless surface with negative Euler characteristic
and a train track folding sequence (equivalently, splitting sequence) for a pseudo-Anosov $f\colon S \to S$, one can construct the veering triangulation of the mapping torus of $f$, whose
flow graph encodes the transition matrix for this train track. Using the work of \cite{AT}, we show that the only obstructions to irreducibility are branches which are dual to infinitesimal cycles of the flow graph and characterize the structure of veering tracks near these branches. This local picture allows us to modify the track by contracting the obstructing branches, and recover a support map for the modified train track with irreducible transition matrix.

\subsection{Connections to literature}\label{discuss}

\par Work of Penner-Papadopoulos \cite[Theorem 4.1]{PP} is sometimes cited as producing $f$-invariant generic train tracks with irreducible transition matrix for a pseudo-Anosov $f$, however it was pointed out to us by Chris Leininger that this construction seems to only produce train tracks which are $f^n$-invariant with irreducible transition matrix for some $n$ possibly greater than one.
\par In the setting of \cite[Theorem 4.1]{PP},
the support map $\sigma$ produced for the $f$-invariant train track $\tau$ is not required to map switches to switches, rather the transition matrix is defined by choosing test points in the interiors of branches and counting the number of times the image of each branch hits each test point. Since switches don't have to map to switches, a branch $b$ of $\tau$ could e.g. be mapped by $\sigma \circ f$ over itself and a small piece of an adjacent branch. Then the transition matrix for $f$ will record that $b$ only maps over itself. However, for some $n$, $(\sigma \circ f)^n(b)$ will be mapped over all of the adjacent branch, so the transition matrix for $(\sigma \circ f)^n$ will record that $b$ maps over itself and the adjacent branch. Hence, it's possible that $f^n(\tau) \prec \tau$ has irreducible transition matrix while $f(\tau) \prec \tau$ does not.

\par Returning to discussion of the present work, we note that much of the theory of train tracks assumes that the tracks considered are \textit{\textbf{generic}}, i.e. trivalent, since this simplification can often be made without loss of generality. In particular, the train tracks associated to veering triangulations are always generic. However, if the flow graph for a pseudo-Anosov $f$ is not irreducible, \Cref{mainthm} necessarily produces train tracks which are not generic. 
This motivates the following question:

\begin{qn}
    For all $f\colon S \to S$ pseudo-Anosov, does there exist a \textit{generic} invariant train track for $f$ with irreducible transition matrix?
\end{qn}

\par \textbf{Acknowledgements.} I would like to thank Sam Taylor for his generous guidance and feedback. I would also like to thank Chi Cheuk Tsang for helpful conversations and comments on a draft of this paper.

\section{Background}\label{background}

\par A \textit{\textbf{strongly connected component}} of a directed graph $G$ is a maximal subset $C$ of the vertices of $G$ such that for any two vertices in $C$, there is a path in $G$ from the first vertex to the second.
If $G$ has only one strongly connected component, it is called \textit{\textbf{strongly connected}}.
A nonnegative square matrix $A$ is called \textit{\textbf{irreducible}} if for every pair of indices $i$, $j$ there is an $n>0$ such that the $(i,j)$-th entry of $A^n$ is positive. $A$ is irreducible if and only if it is the adjacency matrix of a strongly connected graph.
\par
    Given an orientable finite-type surface $S$ without boundary and negative Euler characteristic, a homeomorphism $f\colon S \to S$ is called \textit{\textbf{pseudo-Anosov}} if there is a pair of transverse measured singular foliations $(\Lambda^u,\mu_u)$ and $(\Lambda^s,\mu_s)$, called the \textit{\textbf{unstable}} and \textit{\textbf{stable}} foliations, and a constant $\lambda > 1$, called the \textit{\textbf{stretch factor}}, such that $f(\Lambda^u)=\Lambda^u$, $f(\Lambda^s)=\Lambda^s$ and $f(\mu^u) = \lambda \mu^u$, $f(\mu^s) = \lambda^{-1} \mu^s$. No nontrivial power of a pseudo-Anosov map fixes the homotopy class of any essential closed curve in $S$. A pseudo-Anosov map is a diffeomorphism except at the singular points of $\Lambda^u$.

    \par The \textit{\textbf{mapping torus of $f$}} is the 3-manifold
    $$ M_f = \faktor{(S \times [0,1])}{ (f(p),0) \sim (p,1)} $$ which fibers over the circle with fiber $S$ and monodromy $f$.

    \par If all of the singularities of the foliations are at the punctures of $S$, we say $f$ is \textit{\textbf{fully punctured}}. We will study general pseudo-Anosovs by deleting the singularities of the foliations to obtain a new surface $S^\circ$ on which the restriction of $f$ is fully punctured.
For full background on pseudo-Anosov homeomorphisms, see Fathi--Laudenbach--Poénaru \cite{FLP}.

\subsection{Train tracks}\label{tts}
\par A \textit{\textbf{train track}} $\tau$ is a closed 1-complex embedded in a surface $S$ with a ``smoothing" at each vertex so that $\tau$ has a well-defined tangent space at each vertex, and $S \setminus \tau$ contains no nullgons, unpunctured monogons or unpunctured bigons. See Penner-Harer \cite{PH}. The vertices of $\tau$ are called \textit{\textbf{switches}} and the edges are called \textit{\textbf{branches}}. If the switches of $\tau$ all have degree three, $\tau$ is called \textit{\textbf{generic}}.

\par Given two train tracks $\tau_1$ and $\tau_2$ on $S$, $\tau_2$ is said to \textit{\textbf{carry}} $\tau_1$ (or $\tau_1$ \textit{\textbf{is carried by}} $\tau_2$), denoted $\tau_1 \prec \tau_2$, if there is a $C^1$ map $\sigma\colon S \to S$ called the \textit{\textbf{support map}} 
such that
    \begin{enumerate}
        \item $\sigma$ is homotopic to the identity,
        \item $\sigma(\tau_1) \subseteq \tau_2$,
        \item for all points $x$ in $\tau_1$,
    $d_x \sigma_{|\tau_1}\colon T_x \tau_1 \to T_{h(x)}\tau_2$ is an isomorphism of tangent spaces,
    \end{enumerate} 
Note that carrying is transitive: if $\tau_1 \prec \tau_2$ and $\tau_2 \prec \tau_3$, then
$\tau_1 \prec \tau_3$. 
We also write $\tau_1 \prec_\sigma \tau_2$ to indicate that $\sigma$ is the support map for $\tau_1 \prec \tau_2$. A carrying map $\sigma$ for $\tau_1 \prec \tau_2$  is called \textit{\textbf{combinatorial}} if for every switch $v$ in $\tau_1$, $\sigma(v)$ is a switch in $\tau_2$.
    \par
    Given $\tau_1$ with branches $\{b^1_j\}_{j=1\ldots n}$ and $\tau_2$ with branches $\{b^2_i\}_{i=1\ldots m}$ such that $\tau_1$ is carried by $\tau_2$ with combinatorial support map $\sigma$,
    the \textit{\textbf{transition matrix}} is the $m \times n$ matrix $T = [t_{i,j}]$ where 
    $t_{i,j}$ is the number of times the image of $b^1_j$ under $\sigma$ passes over $b^2_i$.
\par
    A \textit{\textbf{train path}} is a $C^1$ immersion $p\colon [0,1] \to \tau$ so that the endpoints $\{0,1\} = \partial[0,1]$ are mapped to switches. We will sometimes conflate $p$ with its image in $\tau$. For $t \in (0,1)$ such that $p(t)$ is a switch, the sub-paths $p_0 = p_{|[0,t]}\colon [0,t] \to \tau$, $p_1 = p_{|[t,1]}\colon [t,1] \to \tau$ are also train paths (after reparametrizing). Similarly, if $p_0$ and $p_1$ are train paths such that $p_0(1)=p_1(0)$ and the differentials $d_1 p_0(\partial_t)$ and $d_0 p_1(\partial_t)$ are either both positive or both negative. The concatenation $(p_0 p_1)\colon [0,2] \to \tau$ is a train path after reparametrizing, hence any train path can be written as $(p_0\ldots p_{n-1})$, where $p_i\colon\left[\frac{i}{n},\frac{i+1}{n}\right] \to \tau$ has image consisting of a single branch. If $p$ is injective we say the train path is \textit{\textbf{embedded}}.
\par
    A \textit{\textbf{fold}} is a particular combinatorial support map $\phi$ for $\tau_1 \prec \tau_2$ which induces a bijection between the branches of $\tau_1$ and $\tau_2$
    except on a particular connected subgraph where it is defined according to either of the two pictures in \Cref{fig:fold}. The three branches of $\tau_1$ which are mapped over the small branch $e'$ in the domain of \Cref{fig:fold} are said to \textit{\textbf{fold to}} $e'$. 
    If there is a $\tau_1 \prec_\phi \tau_2$
    we say $\tau_1$ \textit{\textbf{folds to}} $\tau_2$, denoted $\tau_1 \leftharpoonup \tau_2$. If $\tau_1$ folds to $\tau_2$, $\tau_1$ is carried by $\tau_2$. If $\tau_1 \prec \tau_2$ with support map given by a sequence of folds we also say $\tau_1$ folds to $\tau_2$.
    \begin{figure}[h]
    \centering
    \includegraphics[width=0.66\textwidth]{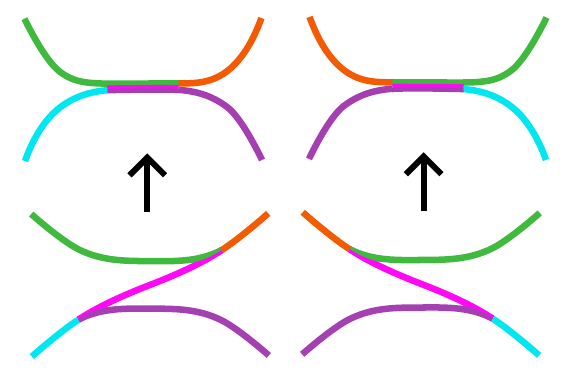}
    \caption{The two possible pictures of a fold. The fold on the left is a left fold and the one on the right is a right fold. Colors indicate the image of each branch on the top after folding.}
    \label{fig:fold}
\end{figure}
    \par
    A train track $\tau$ is an \textit{\textbf{invariant}} track for a map $f: S \to S$ with support map $\sigma$ if $f(\tau)$ is carried by $\tau$ with support map $\sigma$.
    If the transition matrix for $f(\tau) \prec_\sigma \tau$ is irreducible, we also say $\tau$ is \textit{\textbf{irreducible}}.
By smoothly isotoping $\tau$ if necessary, we may assume it is disjoint from the singularities of $\Lambda^u$. This will be convenient later when we remove the singularities of $\Lambda^u$. The following theorem due to Agol is the starting point of the construction of the veering triangulation of the mapping torus of a pseudo-Anosov $f$:

\begin{thm}[\protect{\cite[Theorem 3.5]{A}}]\label{ttfs} Given $f\colon S \to S$ pseudo-Anosov, there exists a generic invariant train track $\tau$ for $f$ such that $f(\tau)$ folds to $\tau$, i.e. there exist a finite sequence of train tracks 
$$f(\tau) = \tau_0 \leftharpoonup \tau_1 \leftharpoonup \ldots \leftharpoonup \tau_n = \tau,$$ 
where $\tau_i$ is carried by $\tau_{i+1}$ with support map given by a single fold.
\end{thm}

\par Note that the referenced theorem is stated in terms of \textit{splits} rather than folds, which are the combinatorial inverse of folds. 
   Additionally, the referenced theorem is stronger than what we state here because the splitting sequence is shown to be canonical up to ``commuting maximal splits": at each step in the sequence, the train track is split at a branch carrying maximal measure of the invariant lamination, and the splits of maximal branches commute if there is not a unique branch carrying maximal measure.
   This theorem motivates the following definition:
\begin{defn}
    An invariant train track $\tau$ for a pseudo-Anosov map $f$ is called \textit{\textbf{veering}} if it 
    is generic and $\tau$ can be obtained from $f(\tau)$ by a sequence of folds, in which case there is a sequence of train tracks 
    $$f(\tau) = \tau_0 \leftharpoonup \tau_1 \leftharpoonup \ldots \leftharpoonup \tau_n = \tau$$ 
    where $\tau_i$ is carried by $\tau_{i+1}$ with support map consisting of a single fold.
\end{defn}
\par In \Cref{clvts} we will see that a veering train track as defined here is sufficient to produce the veering triangulation of the fully punctured mapping torus of $f$. While veering train tracks are not necessarily irreducible, they will be the beginning of our construction of irreducible tracks.
\par
    Given two train tracks $\tau_1$, $\tau_2$ on $S$, a (surface) \textit{\textbf{train track map}} is a map $t: S \to S$ so that $t(\tau_1)$ is contained in $\tau_2$, and the restriction of $t$ to $\tau_1$ is 
    such that for any train path $p\colon [0,1] \to \tau_1$, $t \circ p\colon [0,1] \to \tau_2$ is also a train path. In other words, a train track map is like a combinatorial support map except that it does not have to be homotopic to the identity map.    
    \par The following lemma will be useful for producing support maps:
\begin{lemma}\label{supplemma}
    Let $t\colon S \to S$ be a train track map taking $\tau_1$ to $\tau_2$, and $g\colon S \to S$ be a diffeomorphism.
    If $t$ is homotopic to $g$, then $g(\tau_1) \prec \tau_2$ with support map given by $t \circ g^{-1}$.
\end{lemma}
\begin{proof}
    Since $t$ and $g$ are homotopic, $t \circ g^{-1}$ is homotopic to the identity. Since $t(\tau_1) \subseteq \tau_2$, $(t \circ g^{-1})(g(\tau_1)) = t(\tau_1) \subseteq \tau_2$. Let $p \in g(\tau_1)$. Since $g$ is a diffeomorphism, the restriction of the differential 
    $$d_p g^{-1}_{|g(\tau_1)}\colon T_p g(\tau_1) \to T_{g^{-1}(p)} \tau_1$$
    is an isomorphism. Since $t$ is a train track map, $d_{g^{-1}(p)} t_{| \tau_1}$ is an isomorphism, so the composition
    $$d_p (t \circ g^{-1})_{|g(\tau_1)} = (d_{g^{-1}(p)} t_{| \tau_1}) \circ (d_p g^{-1}_{|g(\tau_1)})\colon T_p g(\tau_1) \to T_{(t \circ g^{-1})(p)} \tau_2$$
    is also an isomorphism.
\end{proof}

\par
A \textit{\textbf{half-branch}} of a branch $b$ in $\tau$ is a component of $b \setminus \{x\}$ for a point $x$ in the interior of $b$ (so a half branch is incident to one switch). For a generic track, we can fix an orientation on the tangent space at a switch $v$ so that 
a single half-branch is incident on the positively oriented side and two half-branches are incident on the negatively oriented side.
The two half-branches on the negatively oriented side are called \textit{\textbf{small}}, and the one on the other side is called \textit{\textbf{large}}. If both components of $b \setminus \{x\}$ are large (resp. small), we also say $b$ is large (resp. small). If $b$ is neither large nor small, we say it is \textit{\textbf{mixed}}.

\par
A \textit{\textbf{patch}} of $S \setminus \tau$ (given $\tau$ on $S$) is a completion of a component of $S \setminus \tau$ under the induced path metric %\cite{L18}.
The non-smooth points of the boundary of a patch are called \textit{\textbf{cusps}}. A patch has a natural map from $P$ to $S$ induced by the embedding of its interior into $S$, we will sometimes conflate patches and their images in $S$. The boundary of a patch consists of train paths between its cusps, called the \textit{\textbf{sides}} of the patch.

Let $P$ be a patch of $S \setminus \tau$, $U$ a connected subset of $S$ intersecting $P$, $\overset{\circ}{P}$ the interior of $P$ as a subset of $S$. A \textit{\textbf{piece of}} $P$ \textit{in} $U$ (or just a \textit{\textbf{piece of a patch}}) is a component of the completion of $U \cap \overset{\circ}{P}$ under the induced path metric for any such $P$ and $U$.

Note that when $\tau$ is an invariant train track for a pseudo-Anosov map, $\tau$ is \textit{\textbf{filling}}, i.e. all the patches of $S \setminus \tau$ are $n$-gons with at most one puncture, as otherwise a patch would contain an essential closed curve fixed by a power of $f$. If $\tau$ is a veering track, patches of $S^\circ \setminus \tau$ all have exactly one puncture because $\tau$ is the spine of an ideal triangulation of $S^\circ$, so patches are dual to ideal vertices (see \Cref{triang}).

\begin{lemma}\label{thm:patches}[Patches to patches, cusps to cusps]
    Let $\tau \subseteq S^\circ$ be a veering train track for $f$, so $f(\tau) \prec_\sigma \tau$ where $\sigma$ is a composition of folds, and let $P$ be a patch of $S^\circ \setminus \tau$. 
    The interior of the image of $P$ under $\sigma \circ f$ is the interior of another patch $Q$ of $S^\circ \setminus \tau$ with the same number of cusps, and $\sigma \circ f$ sends cusps of $P$ to cusps of $Q$. The image of $\partial P$ consists of $\partial Q$ and train paths $q_i$ where $q_i(0)$ is the branch at the $i$-th cusp of $Q$.    
\end{lemma}

\begin{proof}
    Folds preserve the topology (i.e. the cusp structure and number of punctures) of patches, so $\sigma$ preserves the topology of patches. As $f$ is a diffeomorphism away from the singularities of its invariant foliations, it also preserves the topology of patches. It often happens that $\sigma \circ f$ identifies pieces of $\partial P$ on opposite sides of a cusp, in which case these pieces map to a train path starting at the cusp and exiting out the small branch incident to the switch at the cusp.
\end{proof}

For a map $g$ and a train track $\tau$ such that $g(\tau) \prec_\sigma \tau$, if patches $P$ and $Q$ of $S \setminus \tau$ are such that the interior of $(\sigma \circ g)(P)$ is the interior of $Q$, we say $P$ \textit{\textbf{is carried to}} $Q$.

\par Let $s$ be a switch of a train track $\tau$, and let $U$ be a disk neighborhood of $s$ so that $\tau \cap U$ is contractible. We say two half branches $b$, $b'$ incident to $s$ are \textit{\textbf{adjacent}} if there is a patch $P$ so that $b$ and $b'$ are both contained in the same component of $P \cap U$. The condition that train tracks have well-defined tangent spaces at switches is equivalent to putting a taut angle structure around each switch, i.e. requiring the angles between adjacent branches to be either $0$ or $\pi$, so that the sum of the angles around a switch is $2 \pi$, as in \Cref{fig:ttanglestruct}. It will be convenient to construct train tracks by defining the smoothing in this way. This formalism is analogous to the case of taut ideal triangulations one dimension up, where angle structures are used by Lackenby to create branched 2-manifolds \cite{L}.

\begin{figure}[h]
    \centering
    \includegraphics[width=0.66\textwidth]{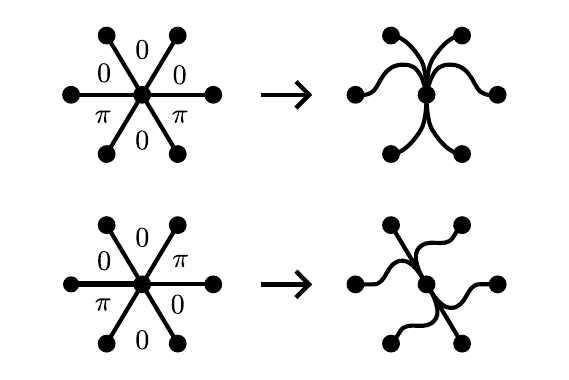}
    \caption{A train track can be obtained from a finite graph embedded in $S$ by imposing a taut angle structure around each vertex, which is equivalent to having a well-defined tangent space at each vertex.}
    \label{fig:ttanglestruct}
\end{figure}

\subsection{Triangulations}\label{triang}
Here we recall some background on triangulations of 3-manifolds which will be important for the construction of the veering triangulation of the mapping torus of $f$ and its relation to train track folding sequences. For additional details see \cite{P} or \cite{L}.

\par Let $\Delta$ be an ideal triangulation of a cusped 3-manifold $M$, in other words a simplicial complex with the vertices deleted which is homeomorphic to $M$.

\par
    $\Delta$ is called \textit{\textbf{taut}} if it admits an angle structure such that each dihedral angle is either $0$ or $\pi$,  around each edge exactly two dihedral angles are $\pi$, and these angles are not adjacent. This gives $\Delta^{(2)}$ the structure of a \textit{\textbf{branched surface}}. 
\par
    Given a surface $S$ and a branched surface $B$ both embedded in $M$,
    $B$ \textit{\textbf{carries}} $S$ if there is a map $\sigma\colon M \to M$ homotopic to the identity such that $h(S) \subseteq B$, and the restriction of the differential of $h$ to $S$ at a point is an isomorphism of tangent spaces. In this case, $\sigma$ is called the \textit{\textbf{carrying map}}.

\par
    The \textit{\textbf{spine}} of $\Delta$ is the dual cell complex $\Delta^*$ which has 
    a 2-cell for every edge of $\Delta$, with edges of $\Delta^*$ identified according to incidences of triangles in $\Delta$. The spine $\Delta^*$ has no 3-cells because $\Delta$ has no material vertices.
    
    Given a cusped surface $S$, an \textit{\textbf{ideal surface triangulation}} of $S$ is a simplicial 2-complex with the vertices deleted $\delta$ homeomorphic to $S^\circ$. If $\delta$ is an ideal triangulation of $S^\circ$, we also associate to $\delta$ a choice of homeomorphism $\iota: S^\circ \to \delta$ called the \textit{\textbf{marking}}.
    Given an ideal surface triangulation $\delta$, the dual complex $\delta^*$ is defined similarly to the case one dimension up, with edges corresponding to edges of $\delta$, vertices corresponding to triangles, and edges identified according to incidence of triangles.
    The 1-complex $\delta^*$ is called the \textit{\textbf{spine}} of $\delta$. $\delta_*$ can be smoothed to a train track by assigning an angle structure at each vertex (provided that the angle structure does not produce a complementary region which is a punctured disk). Conversely, if a generic filling train track $\tau$ on $S$ has exactly one puncture in every patch, $\tau$ is dual to an ideal triangulation of $S$ in the sense that there is an ideal triangulation $\delta$ and a marking $\iota:S \to \tau$ such that $\iota(\tau)$ is the spine of $\delta$.
    \par A \textit{\textbf{flip}} in $\delta$ is the combinatorial move in adjacent 2-cells of $\delta$ replacing their common edge with a new edge between the vertices not adjacent to the common edge. If two generic filling train tracks differ by a fold, their dual triangulations differ by a flip in the edge dual the small branch being folded, as in \Cref{fig:trackdualtriang}. 
    
    \begin{figure}[h]
    \centering
    \includegraphics[width=0.66\textwidth]{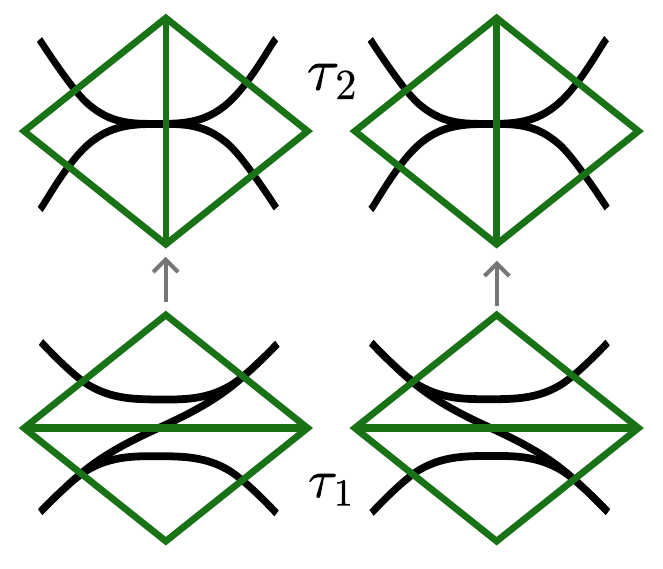}
    \caption{A local picture of the duality between generic filling train tracks and triangulations of $S$. If $\tau_1$ folds to $\tau_2$, the dual triangulations differ by a flip.}
    \label{fig:trackdualtriang}
    \end{figure}

    \par Once and for all, let $S$ be a surface of negative Euler characteristic without boundary, and $f\colon S \to S$ a pseudo-Anosov homeomorphism with unstable foliation $\Lambda^u$. Let $\{v_i\}$ be the singular points of $\Lambda^u$ and $\Lambda^s$, and define $S^\circ = S \setminus \{ v_i \}$. We will also use $f$ to refer to its restriction to $S^\circ$. Let $\tau$ be a veering train track for $f$, so $f(\tau) \prec_\sigma \tau$ with carrying map $\sigma$ consisting of a composition of folds $\sigma = \phi_{n}\circ\ldots\circ \phi_2 \circ\phi_1$.
\section{Dynamics of veering train tracks}\label{dvts}
\subsection{Constructing layered veering triangulations}\label{clvts}
Following \cite{A} and \cite{AT}, we will now use the veering track $\tau$ to construct the \textit{(layered) veering triangulation} of $M_f$.
\par To begin, there is an ideal surface triangulation $\delta_0$ and a marking $\iota_0\colon S^\circ \to \delta_0$ such that $\tau_0 = \iota_0(f(\tau))$ is the spine of $\delta_0$.
Let $b_0$ be the small branch of $f(\tau)$ that is folded by $\phi_1$, and let $e_0$ be the edge of $\delta_0$ dual to $\iota_0(b_0)$.
Attach an ideal tetrahedron $T_1$ to $\delta_0$ by gluing two of its faces to the two faces of $\delta_0$ adjacent to $e_0$, forming a new simplicial complex $\Delta_1$ as in \Cref{fig:gluedelta0}.
Put a taut structure on $\Delta_1$ by assigning dihedral angle $\pi$ to $e_0$ and the edge of $T_1$ not in $\delta_0$, and angle $0$ to the remaining four edges in $T_1$, and coorient $T_1^{(2)}$ so that $e_0$ is on the bottom and the edge of $T_1$ not in $\delta_0$ is on the top.
Then there is a second triangulation $\delta_1$ on the top of $\Delta_1$ that differs from $\delta_0$ by a flip, and a marking $\iota_1 \colon S^\circ \to \delta_1$ so that $\iota_1(\phi_1(f(\tau)))$ is the spine of $\delta_1$ and $\iota_1$ agrees with $\iota_0$ on $\iota_0^{-1}(\delta_0 \cap \delta_1)$.

\par Continue inductively, defining a simplicial complex $\Delta_k$ by attaching a tetrahdron $T_{k}$ to $\Delta_{k-1}$ by gluing two faces of $T_{k}$ to the two faces of $\Delta_{k-1}$ adjacent to the edge $e_{k-1}$ dual to the branch of $\tau_{k-1} = \iota_{k-1}(\phi_{k-1}\circ\phi_{k-2}\circ\ldots\circ\phi_1(f(\tau))$ which is the image of the small branch of $\tau$ folded by $\phi_{k-1}$. 
Assign a taut structure to $\Delta_k$ as above, assigning dihedral angle $\pi$ to $e_{k-1}$ and the single new edge in $\Delta_k \setminus \Delta_{k-1}$, and angle $0$ to the remaining interior angles of $T_{k}$, and coorienting so $\delta_{k-1}$ is on the bottom of $T_{k}$ and the faces of $\Delta_k \setminus \Delta_{k-1}$ are on the top.
Then there is a new surface triangulation $\delta_k$ on the top of $\Delta_k$ and a marking $\iota_k\colon S^\circ \to \delta_k$ so that $\tau_{k} = \iota_{k}(\phi_{k}\circ\phi_{k-1}\circ\ldots\circ\phi_1(f(\tau))$ is the spine of $\delta_k$ and $\iota_k$ agrees with $\iota_{k-1}$ on $\iota_{k-1}^{-1}(\delta_{k-1} \cap \delta_k)$.

\par The \textit{\textbf{cut open veering triangulation}} is the simplicial complex $\Delta \ca$ $(=\Delta_n)$ obtained at the end of this process. %See \Cref{fig:cartoon} for a summary of the maps defined in this section.

\par $\Delta\ca$ is almost a triangulation of $S^\circ \times [0,1]$, except that if a branch $b$ of $f(\tau)$ is never folded, there aren't any tetrahedra above the edge of $\delta_0$ dual to $\iota_0(b)$.
Since $f(\tau)$ and $\tau$ differ by a homeomorphism and the $\delta_i$ are dual to the $\tau_i$, the triangulations $\delta_0$ and $\delta_n$ on the bottom and top of $\Delta\ca$ are isomorphic.
By construction, the homeomorphism $\iota_0 \circ f \circ \iota_n^{-1}\colon \delta_n \to \delta_0$ restricts to a homeomorphism from $\tau_n$ to $\tau_0$. We wish to promote this map to a simplicial map while still sending $\tau_n$ to $\tau_0$.
To do this, observe that given an edge $e$ of $\delta_n$, $\iota_0^{-1}(e)$ is an essential arc in $S^\circ$ dual to a branch $b$ of $\tau$, and $f(b)$ is also dual the essential arc $\iota_n^{-1}(b')$ for some edge $e'$ of $\delta_0$, so $\iota_0^{-1}(e)$ and $f^{-1} \circ \iota_n^{-1}(e')$ are homotopic.
The arc $\iota_n^{-1}(e)$ cannot intersect any edge of $f^{-1} \circ \iota_0^{-1}(\delta_0)$ other than $f^{-1} \circ \iota_0^{-1}(e')$ because otherwise $\tau$ could not be dual to both $f^{-1} \circ \iota_0^{-1}(\delta_0)$ and $\iota_n^{-1}(\delta_n)$. Hence, there is an open neighborhood of 
$f^{-1} \circ \iota_0^{-1}(b') \cup \iota_n^{-1}(b)$
which is disjoint from $\tau \setminus e$ and the rest of $f^{-1} \circ \iota_0^{-1}(\delta_0)$ and $\iota_n^{-1}(\delta_n)$. Thus $f^{-1} \circ \iota_0^{-1}(b')$ and $\iota_n^{-1}(b)$ are homotopic through an ambient isotopy which is constant on the complement of this disk and fixes $b$ setwise. Hence, by composing $\iota_n$ with this isotopy we may assume that $\iota_n^{-1}(b) = f^{-1} \circ \iota_0^{-1}(b')$, or equivalently that $\iota_0 \circ f \circ \iota_n^{-1}(b) = b'$. By doing this for each edge of $\delta_n$, we may assume that $\iota_0 \circ f \circ \iota_n^{-1}$ is a simplicial map sending $\tau_n$ to $\tau_0$.

\begin{figure}[h]
    \centering
    \includegraphics[width=0.4\textwidth]{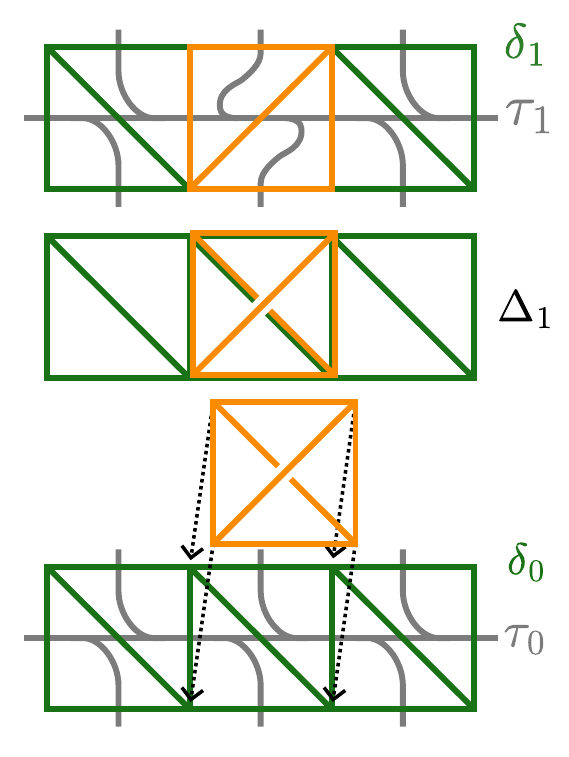}
    \caption{The fold from $\tau_0$ to $\tau_1$ is realized by placing a tetrahedron on $\delta_0$ above the branches being folded, so that the surface triangulation $\delta_0$ dual to $\tau_0$ is on the bottom of the resulting simplicial complex $\Delta_1$ and the surface triangulation $\delta_1$ dual to $\tau_1$ is on the top.}
    \label{fig:gluedelta0}
\end{figure}
\par The \textit{\textbf{veering triangulation}} $\Delta$ of $M_f$ is then defined by 
quotienting $\Delta \ca$ by this map:
$$\Delta = \faktor{\Delta\ca}{p \sim (\iota_0 \circ f \circ \iota_n^{-1})(p) \:\forall p \in \delta_n}.$$
$\Delta$ inherits the taut structure and dualities of $\Delta\ca$.

\par The edges of $\Delta$ admit a 2-coloring as follows. Observe that every tetrahedron in $\Delta$ has exactly one bottom edge. Conversely, every edge in $\Delta$ is the bottom edge of exactly one tetrahedron, and passing from the bottom faces to the top faces in this tetrahedron affects a fold in the dual train tracks from the bottom to the top of the tetrahedron. Color this bottom edge red if the fold is a right fold and blue if it is a left fold as in \Cref{fig:veeringcolor}. See e.g. \cite[Definition 2.2]{AT}. Since every edge is the bottom edge of exactly one tetrahedron, this determines the coloring of every edge. In each tetrahedron, opposite side edges are the same color and adjacent side edges are the opposite color as in \Cref{fig:veeringtet}.

\begin{figure}[h]
    \centering
    \includegraphics[width=0.66\textwidth]{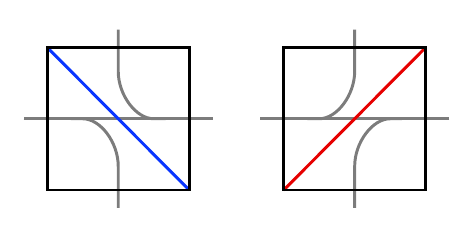}
    \caption{The bottom edge of a tetrahedron is colored according to whether the tetrahedron is affecting a left fold or a right fold.}
    \label{fig:veeringcolor}
\end{figure}

\par Every switch of a train track in $\Delta$ is dual to a triangle, which is a bottom face of exactly one tetrahedron, bounded by one bottom edge and two side edges. 
As seen in \Cref{fig:veeringcolor}, the branch dual to the bottom edge is on the right side of the cusp (resp. left side) at this switch if and only if the bottom edge of the triangle dual to the switch is colored blue (resp. red). 
Since every edge of $\Delta$ is the bottom edge of exactly one tetrahedron, an equivalent characterization of the coloring is that 
edges dual to branches on the right (resp. left) side of a cusp of a train track are colored blue (resp. red).

\begin{figure}[h]
    \centering
    \includegraphics[width=0.75\textwidth]{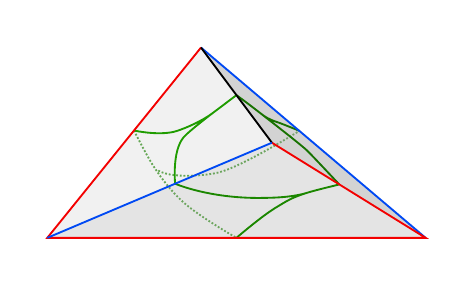}
    \caption{The coloring of a single veering tetrahedron and the train tracks on the top and bottom. The coloring on the bottom edge depends whether the fold affected by the tetrahedron is a right fold or a left fold, and the coloring of the top edge is determined by the fold in the tetrahedron where this edge is on the bottom.}
    \label{fig:veeringtet}
\end{figure}

\par Let $q \colon \Delta\ca \to \Delta$ be the quotient map. 
We can define a carrying map $\xi \colon S^\circ \to \Delta^{(2)}$ by $\xi = (q \circ \iota_n)$.
Note that $\xi$ is not necessarily an embedding because if an edge $e$ of $\delta_0$ has no tetrahedra above it in $\Delta\ca$, it is also in $\delta_n$, so it will be identified with $\iota_0 \circ f \circ \iota_n^{-1}(e)$.
\par The train track $q(\tau_n)$ on $\Delta^{(2)}$ pulls back under $\xi$ to the veering train track $\tau$ for $f$ that we started with. 
In addition, $\Delta\ca$ admits a coloring by pulling back the coloring of $\Delta$.
This carrying map $\xi$ will continue to be important for pulling back structures in $\Delta$ to $S^\circ$.

\subsection{The flow graph and the transition matrix}\label{flowgraph}

Following \cite{AT} and Landry--Minsky--Taylor \cite[section 2]{LMT1}, we will now construct a directed graph embedded in $\Delta\ca$ by an inductive process analogous to the construction of $\Delta\ca$ which records which branches are folded. 
Let $\Phi_0$ be the graph embedded in $\delta_0$ which has a vertex at each point of intersection between $\delta_0$ and $\tau_0$ and no edges. Let $\Phi_1$ be the graph obtained from $\Phi_0$ by adding a single vertex at the point of intersection between $\delta_1$ and the branch of $\tau_1$ which has three branches of $\tau_0$ folded onto it, and an edge from each of the three vertices on these branches of $\tau_0$ to the new vertex. Continue inductively, defining $\Phi_{i+1}$ by adding a single new vertex to $\Phi_i$ at the point of intersection between the branch of $\tau_{i+1}$ which has three branches of $\tau_i$ folded onto it and edges from these three branches to the new vertex as in \Cref{fig:cofg1}. Observe that in terms of the coloring of the edges of the triangulation, this means that in each tetrahedron there are edges from the side edges of opposite color to the bottom edge to the top edge, and an edge from the bottom edge to the top edge.
\par Define the \textit{\textbf{cut open flow graph}} to be the graph $\Phi\ca$ obtained at the end of this process. Define the \textit{\textbf{flow graph}} $\Phi$ to be the quotient $q(\Phi \ca)$ of $\Phi \ca$ in $\Delta$.

\begin{figure}[h]
    \centering
    \subfigure[]{\includegraphics[width=0.6\textwidth]{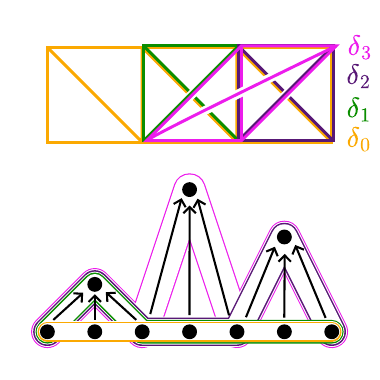}}
    \subfigure[]{\includegraphics[width=0.39\textwidth]{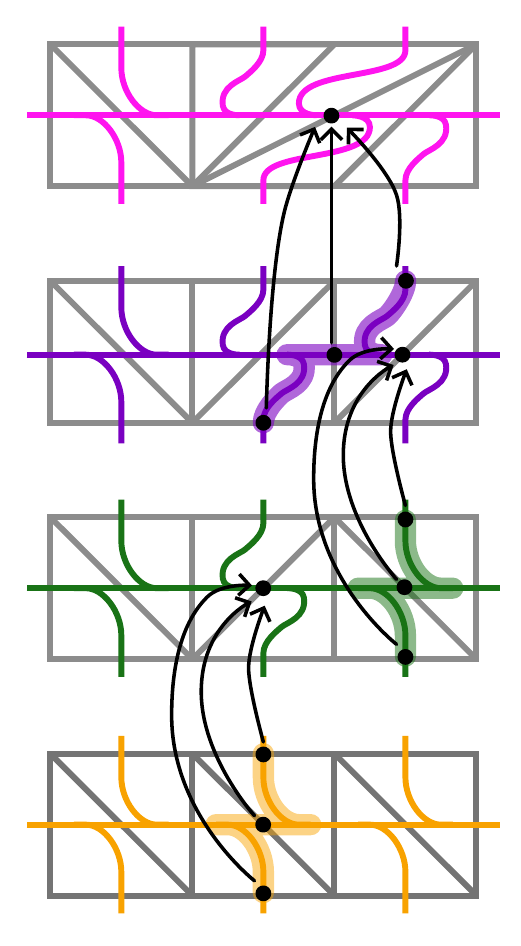}}
    \caption{(a) An example of a piece of the cut open flow graph $\Phi_n$ as an abstract directed graph, and a top down view of $\Delta_3$. (b) An exploded view of the edges between each $\delta_i$ and $\delta_{i+1}$. Edges of $\Phi_n$ go from vertices on the branches of $\tau_i$ being folded to the vertex on the branch of $\tau_{i+1}$ which is in the image of all three. Hence, a path in $\Phi_n$ tracks a piece of the image of a branch of $f(\tau)$ under the folding sequence $\sigma$.}
    \label{fig:cofg1}
\end{figure}

\par Finally, we want to relate the transition matrix for $f(\tau) \prec_\sigma \tau$ to paths in the cut open flow graph. Let $b$ be a branch of $\tau$ and let $(b_1,\ldots,b_m)$ be the train path constituting $\sigma(f(b))$ (so note that we might have $b_j = b_k$ for $j \ne k$). In $\Delta \ca$, the branch $f(b)$ of the train track $f(\tau)$ corresponds to the branch $\iota_0(b)$ in $\tau_0$. The train path $(b_1 \ldots b_m)$ in $\tau$ corresponds to $(\iota_n(b_1)\ldots \iota_n(b_m))$ in $\tau_n$.

\begin{lemma}\label{fgpaths1}
    For a branch $b$ of $\tau$ with
    $(\sigma \circ f)(b) = (b_1,\ldots,b_m)$
    allowing $b_j = b_k$ as above
    there is exactly one path in the cut open flow graph (including paths of length zero) from the vertex corresponding to $\iota_0(b)$ to the vertex corresponding to $\iota_n(b_j)$ and no paths from $\iota_0(b)$ to $\delta_n$ other than these.
\end{lemma}
\begin{proof}
    \par We induct on the length of the folding sequence $\phi_k \circ \ldots \circ \phi_1$. For the base case $k = 0$, $b$ is not folded, there is a path of length zero from $\iota_0(b)$ to itself and no other paths from $\iota_0(b)$ to $\delta_0$.

    \par Suppose $(\phi_k \circ \ldots \phi_1 \circ f)(b) = (b_1,\ldots, b_m)$, there is one path in $\Phi \ca$ from (the vertex corresponding to) $\iota_0(b)$ to (the vertex corresponding to) $\iota_k(b_j)$ and no paths from $\iota_0(b)$ to $\delta_k$ other than these.
    If $b_j$ is not one of the three branches folded by $\phi_{k+1}$ then $b_j$ is also in $\phi_{k+1} \circ \phi_k \circ \ldots \circ \phi_1(b)$ and the edge dual to $\iota_k(b_j)$ is in both $\delta_k$ and $\delta_{k+1}$, so the path from $\iota_0(b)$ to $\iota_k(b_j)$ is also a path from $\iota_0(b)$ to $\delta_{k+1}$ and there are no other paths from $\iota_0(b)$ to $\delta_{k+1}$.
    
    \par If the branch $b_j$ of $(\phi_k \circ \ldots \circ \phi_1 \circ f)(\tau)$ is the large branch folded by $\phi_{k+1}$, it is dual to the bottom edge of the unique tetrahedron $T_{k+1}$ in $\Delta_{k+1}$ which is not in $\Delta_k$, $\phi_{k+1}(b_j)$ consists of the single branch $\iota_{k+1}(\phi_{k+1}(b_j))$ which is dual to the top edge of $T_{k+1}$, and there is an edge of $\Phi \ca$ from the bottom edge of $T_k$ to the top edge. Hence, there is a path in $\Phi \ca$ from $\iota_0(b)$ to $\iota_{k+1}(\phi_{k+1}(b_j))$, 
    and there are no other paths from $\iota_k(b_j)$ to $\delta_{k+1}$.
    \par If $b_j$ is one of the other branches folded by $\phi_{k+1}$, it is dual to a side edge of $T_{k+1}$, and $\iota_{k+1}(\phi_{k+1}(b_j))$ consists of the two branches of $\iota_{k+1}(\phi_{k+1} \circ \ldots \phi_1(f(\tau)))$ which are dual to the top edge of $T_{k+1}$ and the side edge dual to $b_j$. There is a path of length zero from the side edge dual to $b_j$ in $\delta_k$ to itself in $\delta_{k+1}$, an edge from this side edge to the top edge, and no other paths from $\iota_k(b_j)$ to $\delta_{k+1}$.
\end{proof}

In $\Delta$, a branch $b$ of $\delta_n$ is identified with the branch $\iota_0 \circ f^{-1} \circ \iota_n^{-1}(b)$ of $\delta_0$. This corresponds to the identification of the branch $\iota^{-1}_n(b)$ in $\tau$ with its image $f(\iota^{-1}_n(b))=\iota_0^{-1}(b)$ in $f(\tau)$. The transition matrix $M$ for $f(\tau) \prec_\sigma \tau$ is defined by indexing the branches of $\tau$ (as opposed to indexing the branches of $f(\tau)$), so to characterize $M$ in terms of paths in $\Phi \ca$ we need to convert the $\delta_0$-indexed branches back to $\delta_n$-indices using the change-of-index $b \mapsto \iota_n \circ f^{-1} \circ \iota_0^{-1}(b)$ from the $\tau_0$-indices to $\tau_n$-indices.
In particular, consider branches $b$, $b'$ in $\tau$ so that $\sigma \circ f(b)$ is carried over $b'$. Then by the above lemma, there is a path in $\Phi \ca$ from $\iota_0(b)$ to $\iota_n(b')$. The branch $\iota_0(b)$ in $\tau_0$ is identified with the branch $(\iota_n \circ f^{-1} \circ \iota_0^{-1})(\iota_0(b)) = \iota_n \circ f^{-1}(b)$. 
\par This motivates defining a graph $G$ whose vertices correspond to those of $\Phi \ca$ in $\delta_0$, and which has an edge $(v,v')$ for every path in $\Phi \ca$ from $v$ in $\delta_0$ to $\iota_n \circ f^{-1} \circ \iota_0^{-1}(v')$ in $\delta_n$ (including paths of length $0$).

\begin{prop}\label{fgpaths}
    The transition matrix $M$ for $f(\tau) \prec_\sigma \tau$ is given by the adjacency matrix of $G$.
\end{prop}
\begin{proof}
    Let $b_i$, $b_j$ be branches of $\tau$. Recall that the $(i,j)$-th entry of $M$ records the number of times $f(b_i)$ is carried of $b_j$, and the $(i,j)$-th entry of the adjacency matrix of $G$ records the number of paths in $\Phi \ca$ from the vertex corresponding to $\iota_0(b_i)$ to the vertex corresponding to $\iota_n \circ f^{-1} \circ \iota_0^{-1}(b_j)$. By \Cref{fgpaths1}, these quantities coincide.
\end{proof}

\subsection{Walls of veering triangulations}\label{wvts}
We now summarize Agol--Tsang's characterization of the strongly connected components of the flow graph.
\begin{defn} \cite[Definition 3.3]{AT}:
    Given $w \ge 2$, $h \ge 1$, a \textit{\textbf{wall}} is a collection of tetrahedra $t_{i,j}$ in a veering triangulation $\Delta$ with $1 \le i \le w+1$, $j \in \ZZ/h$ such that for $2 \le i \le w$,

    \begin{enumerate}
        \item The bottom edge of $t_{i,j}$ is the top edge of $t_{i,j+1}$
        \item  If $i$ is odd, the bottom edge of $t_{i,j}$ is a side edge of $t_{i-1,j}$ and $t_{i+1,j}$ and is not a side edge of any other tetrahedra in $\Delta$. If $i$ is even, the bottom edge of $t_{i,j}$ is a side edge of $t_{i-1,j+1}$ and $t_{i+1,j+1}$ and no other tetrahedra in $\Delta$.
    \end{enumerate}

    $w$ is called the \textit{\textbf{width}} of the wall and $h$ is called the \textit{\textbf{height}}. We will assume that walls are \textit{\textbf{maximal}}, i.e. not properly contained in any other wall of greater width.
\end{defn}

For $2 \le i \le w$, there is a cycle $c_i$ in $\Phi$ which intersects the top edge of $t_{i,j}$ for every $j$ and no other edges. It follows that there are no paths from $c_i$ to the rest of the flow graph, so we call these cycles \textit{\textbf{infinitesimal}}. There are also cycles $c_1$ and $c_{w+1}$ associated to the stacks of tetrahedra $\{t_{1,j}\}$ and $\{t_{w+1,j}\}$ which we call \textit{\textbf{boundary}} cycles. If the top edge of a tetrahedra $t_{i,j}$ in a wall is dual to a vertex lying in a boundary cycle (so $i=1$ or $w+1$), we also call it a \textit{\textbf{boundary}} tetrahedron. If not, we call it an \textit{\textbf{infinitesimal}} tetrahedron. 

\begin{prop}\label{fg sccs}
    \cite[Theorem 3.5]{AT}: The strongly connected components of $\Phi$ consist of the infinitesimal cycles and one ``reduced'' subgraph $\Phi_{red}$.
\end{prop}

Let $A$ be the adjacency matrix of $\Phi$, and let 
%Let 
$A_{red}$ be the square submatrix of $A$ giving the adjacency matrix of the subgraph $\Phi_{red}$ of $\Phi$. 
Let $\Phi_{\red} \ca = q^{-1}(\Phi_\red)$ and let $G_\red$ be the subgraph of $G$ induced by the vertices corresponding to those of $\Phi_{\red} \ca$ in $\delta_0$. Since $\Phi_\red$ is strongly connected, $A_\red$ is irreducible.

\begin{cor}
    If $\Phi$ is not strongly connected, it admits a labeling so that its adjacency matrix $A$ is of the form

    $$A=\begin{bmatrix}
        A_{\red}  & 0 \\
        B & S
    \end{bmatrix}$$

    Where $B$ corresponds to edges from $\Phi_{\red}$ into infinitesimal cycles and $S$ is a permutation matrix.
\end{cor}

\begin{lemma}\label{gred irr}
    \cite[section 4]{AT}: $G_{\red}$ is strongly connected, hence the adjacency matrix of $G_\red$ is irreducible.
\end{lemma}

Here we summarize some observations about the structure of walls that we will need, which are rephrasings or immediate consequences of \cite[section 3]{AT}:
\begin{prop}\label{wallprop}
    \begin{enumerate}
        \item For a wall $W$ of width $w \ge 3$, the edges of $\Delta$ which are dual to infinitesimal and boundary cycles in $W$ all have the same color,
        and the rest of the edges bounding tetrahedra in $W$ have the opposite color.

        \item Every triangle bounding a tetrahedron in a wall has two edges which are dual to infinitesimal or boundary cycles, and one which is not.

        \item Maximal walls $\{t_{i,j}\}$, $\{t'_{i',j'}\}$ of widths $w$, $w'$ are disjoint except possibly at their boundary tetrahedra $t_{1,*}$, $t_{w+1,*}$ and $t'_{1,*}$, $t'_{w'+1,*}$, so no infinitesimal cycle (or infinitesimal tetrahedron) is contained in more than one wall.
        
        \item A tetrahedron appears in walls at most twice. If a tetrahedron appears twice in walls it is a boundary tetrahedron. If a boundary tetrahedron appears twice in walls, there are exactly two infinitesimal cycles meeting its side edges, and the side edges meeting infinitesimal cycles are not adjacent.
    \end{enumerate}
\end{prop}
\begin{proof}
    In \cite[section 3]{AT} it is observed that tetrahedra in a wall of width at least three all have top and bottom edges of the same color (they are all ``fan tetrahedra" of the same color). Since two of the side edges of $t_{i,j}$ are the top edges of other tetrahedra in the wall, they have the same color as the top edges. Since adjacent side edges in a veering tetrahedron have opposite color, the remaining two side edges of $t_{i,j}$ have the opposite color.

    \par Every triangle in a taut tetrahedron is bounded by two adjacent side edges and either the top or the bottom edge.
    Every top edge is dual to an infinitesimal or boundary cycle, and one of the side edges is the top edge of a different tetrahedron in the wall. The third edge is not dual to an infinitesimal or boundary cycle of the wall because it is the opposite color.
    \par The third observation is made explicitly in \cite[section 3]{AT}, as is the observation that a tetrahedron appears in walls at most twice, in which case it must be a boundary tetrahedron. 
    A side edge $e$ of a boundary tetrahedron with opposite color to the bottom edge cannot correspond to a vertex in an infinitesimal cycle because there are at least two edges out of the corresponding vertex-- one from the bottom to the top of the tetrahedron for which $e$ is the bottom edge, and one from $e$ to the top of the boundary tetrahedron. Hence, infinitesimal cycles meeting the side edges of a boundary tetrahedron must pass through side edges of the same color as the bottom edge, which are on opposite sides of the tetrahedron.
\end{proof}

\subsection{Infinitesimal branches of veering train tracks}\label{ibvtts}
In this section we use what we know about infinitesimal cycles of $\Phi$ to characterize the structure of $\tau$ around branches which obstruct irreducibility.

    By abuse of notation denote by $f_{|\tau}$ the restriction of $(\sigma \circ f)$ to $\tau$. A branch $b$ of $\tau$ is called \textit{\textbf{infinitesimal}} if $f^n_{|\tau}(b) = b$ for some $n > 0$, i.e. if for some power of $f$, $b$ is only carried over itself and no other branches. 

\begin{prop}\label{infbranchprop}
    A branch of $\tau$ is infinitesimal if and only if its image in $\Delta$ is dual to the top edge (or equivalently the bottom edge) of a tetrahedron in a wall.
\end{prop}
\begin{proof}
Suppose $f^n_{|\tau}(b_0) = b_0$. Then for every $i$, $f^i_{|\tau}(b_0)$ must consist of exactly one branch $b_i$ since $f_{|\tau}$ is a train track map. By \Cref{fgpaths}, there is a path in the cut open flow graph from (the vertex corresponding to) $\iota_0(b_i)$ to (the vertex corresponding to) $\iota_n(b_{i+1})$, and no path from $\iota_0(b_i)$ to any other branch in $\delta_n$. Therefore the composition of these paths is a cycle in $\Phi$ with no paths back to the rest of the graph. Hence this cycle constitutes a strongly connected component of the flow graph, which is an infinitesimal cycle by \Cref{fg sccs}, and $\iota_0(b_0)$ is dual to an edge of $\Delta$ which intersects this cycle.
\par Conversely, consider an infinitesimal cycle $C \subseteq \Phi$. 
The image $\xi(S^\circ)$ of $S^\circ$ in $\Delta$ intersects $C$ in some edge, which is dual to a branch of $\tau_0$.
By again applying \Cref{fgpaths}, we see that $b$ is infinitesimal since it only maps over itself.
\end{proof}

\par We say that a branch of $\tau$ is a \textit{\textbf{boundary}} branch if its dual edge in $\Phi$ is in a boundary cycle. If a branch is incident to an infinitesimal branch and it is not a boundary branch, we say it is a \textit{\textbf{buttress}} branch, as in \Cref{fig:branchtypes}. By part one of \Cref{wallprop}, in a wall of width at least three boundary branches of a given wall are dual to edges of the same color as those dual to the infinitesimal edges of the wall, and buttress branches are dual to edges of the opposite color.

\begin{figure}[h]
    \centering
    \includegraphics[width=0.66\textwidth]{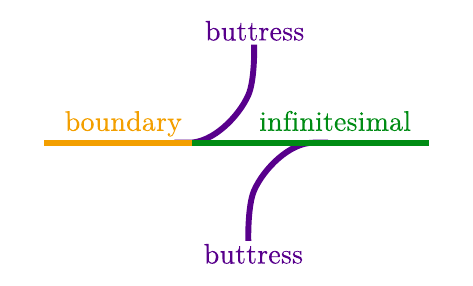}
    \caption{Different types of branches in $N(I_\alpha)$ near an endpoint of $I_\alpha$.}
    \label{fig:branchtypes}
\end{figure}

\begin{lemma}\label{thm:infbranch}
    Boundary and buttress branches of $\tau$ are not infinitesimal.
\end{lemma}
\begin{proof}
    \par We will show that the only infinitesimal branches incident to a given infinitesimal branch are those dual to the top edges of infinitesimal tetrahedra in the same wall. Since boundary and buttress branches are incident to infinitesimal branches, and are not dual to the top edges of infinitesimal tetrahedra, this suffices to show that boundary and buttress branches are not infinitesimal. 
    \par Let $i_\alpha$ be an infinitesimal branch of $\tau$, so $\xi(i_\alpha)$ is dual to the top edge of a tetrahedron $t(i_\alpha)$ in a maximal wall $W_\alpha$ of width $w_\alpha$. 
    We consider two cases, depending on whether $w_\alpha = 2$ or $w_\alpha \ge 3$. 
    Let $v$ be one of the switches incident to $i_\alpha$ and $b$ the buttress branch incident to $v$. The third branch incident to $v$ is either an infinitesimal or boundary branch of $W_\alpha$.
    
    \par If $w_\alpha = 2$, there is a single stack infinitesimal tetrahedra in $W_\alpha$, and $\xi(i_\alpha)$ is dual to the top edge of one of the tetrahedra in this stack. We now show that none of the branches incident to $i_\alpha$ are infinitesimal. The infinitesimal branch $i_\alpha$ is incident to boundary branches $b_1$ and $b_2$ on both sides, which are dual to edges on the top of boundary tetrahedra $t(b_1)$, $t(b_2)$ in $W_\alpha$. If $t(b_1)$ and $t(b_2)$ do not appear a second time in a wall, $b_1$ and $b_2$ cannot be infinitesimal branches, and the buttress branches incident to $i_\alpha$ are not infinitesimal. The infinitesimal tetrahedron $t(i_\alpha)$ cannot appear a second time in a wall by part three of \Cref{wallprop}, so $i_\alpha$ is not adjacent to any other infinitesimal branches. If $t(b_1)$ or $t(b_2)$ appear a second time in a wall, one might be concerned that a buttress branch incident to one of the boundary branches is an infinitesimal branch of the other wall, but this cannot happen by part four of \Cref{wallprop}. Hence, we again conclude that $i_\alpha$ is not adjacent to any other infinitesimal branch when $w_\alpha = 2$.
    
    \par Now we consider the case $w_\alpha \ge 3$. For now suppose the edge dual to $\xi(i_\alpha)$ is blue, so the edge dual to $\xi(b)$ is red.
    If $\xi(b)$ was dual to an edge in an infinitesimal cycle in some other wall $W_\beta$, $t(i_\alpha)$ would have to be a boundary tetrahedron of both walls by part three of \Cref{wallprop}, and so there would be a branch $b'$ of $\tau$ where $\xi(b')$ is dual to a boundary cycle in both walls.
    However, this cannot occur either because $\xi(b')$ being a boundary branch of $W_\alpha$ forces it to be dual to a blue edge since $i_\alpha$ is, while being a boundary branch of $W_\beta$ forces it to be red since $b$ is. 
    If the edge dual to $\xi(i_\alpha)$ is blue, the same argument holds interchanging the colors. We conclude that for $w_\alpha \ge 3$ buttress branches of $W_\alpha$ cannot be infinitesimal.

    \par Similarly, if $b$ is a boundary branch such that $\xi(b)$ is dual to the top edge of a boundary tetrahedron $t(b)$, $b$ cannot also be infinitesimal because otherwise $t(b)$ would also be an infinitesimal tetrahedron, contradicting part three of \Cref{wallprop}, so boundary branches also cannot be infinitesimal for $w_\alpha \ge 3$.
\end{proof}

\begin{lemma}\label{thm:infcomp}
    Let $I \subseteq \tau$ be the subgraph of all infinitesimal branches of $\tau$.
    \begin{enumerate}
        \item The subgraph $I$ is a finite disjoint union of embedded train paths $\{I_\alpha\}_{\alpha = 1 \ldots A}$.
        \item There is an open neighborhood $N(I)$ of $I$ in $S^\circ$ which is a disjoint union of open disks $N(I_\alpha)$ around the train paths $I_\alpha$, and which does not intersect branches of $\tau$ which are not incident to $I$.
        \item The intersection of $\tau$ and $N(I_\alpha)$ is contractible in $S$.
    \end{enumerate}
    
\end{lemma}
\begin{proof}
    Let $i_\alpha$ be an infinitesimal branch of $\tau$, so $\xi(i_\alpha)$ is dual to the top edge of a tetrahedron $t(i_\alpha)$ in a maximal wall $W_\alpha$ of width $w_\alpha$, and let $I_\alpha$ be the connected component of $I$ containing $I_\alpha$. By \Cref{infbranchprop}, each branch in $\xi(I_\alpha)$ is dual to an edge in an infinitesimal cycle of $\Phi$. We will show that $I_\alpha$ is actually an embedded train path.
    \par First, observe that $I_\alpha$ cannot contain a graph cycle because if it did, the cycle would be essential in $S^\circ$ and fixed by a power of $f$ since infinitesimal branches are fixed by a power of $f$ by definition.

    \par Since by \Cref{thm:infbranch} the only infinitesimal branches $i_\alpha$ is incident to are other infinitesimal branches of $W_\alpha$, and $I_\alpha$ cannot contain a cycle, $I_\alpha$ is a path as a subgraph of $\tau$. We now show that this graph path is actually a smooth train path. If $w_\alpha = 2$ this is clear because $I_\alpha = (i_\alpha)$. For $w_\alpha \ge 3$, if $i_\alpha$ is small at $v$, the buttress branch incident to $v$ is also small at $v$, so the third branch is large. If $i_\alpha$ is large at $v$, the other two branches are small, so in either case there is a train path through $v$ containing $i_\alpha$ and the third branch. If the third branch is a boundary branch (so it is not infinitesimal), $v$ has a neighborhood in $I_\alpha$ diffeomorphic to $[0,1)$. If the third branch is an infinitesimal branch, $v$ has a neighborhood in $I_\alpha$ diffeomorphic to $(-1,1)$. Since $v$ is an arbitrary switch of the subgraph $I_\alpha$, we conclude that $I_\alpha$ is actually train path.
    
    \par There are only finitely many such train paths because $\tau$ has finitely many branches, and every infinitesimal branch is contained in some such train path since $i_\alpha$ was chosen arbitrarily. Since the $I_\alpha$ are isolated from each other, they are contained in disjoint neighborhoods $N(I_\alpha)$, and the neighborhoods can be taken such that $\tau \cap N(I_\alpha)$ contains only $I_\alpha$ and infinitesimal and buttress half-branches, so in particular $\tau \cap N(I_\alpha)$ is contractible.
\end{proof}

See \Cref{fig:NI} for an example picture of $N(I_\alpha)$.

\begin{figure}[h]
    \centering
    \includegraphics[width=0.75\textwidth]{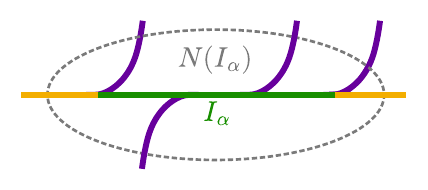}
    \caption{An infinitesimal train path $I_\alpha$ and the branches incident to it. The neighborhood $N(I_\alpha)$ is disjoint from branches not incident to $I_\alpha$ and neighborhoods of other components of $I$.}
    \label{fig:NI}
\end{figure}

Henceforth, let $\{I_\alpha = (i^\alpha_2 \ldots i^\alpha_{w_\alpha})\}_{\alpha=1\ldots A}$ be the maximal train paths consisting of the infinitesimal branches. Let $\{N(I_\alpha)\}_{\alpha = 1\ldots A}$ be pairwise disjoint neighborhoods of the $I_\alpha$ in $S$ which are disjoint from branches which are not incident to $I_\alpha$.
Then $I$ is the union of the $I_\alpha$ and $N(I)$ is the union of the $N(I_\alpha)$. 
Let $v^\alpha_{j+1}$ be the switch between $i^\alpha_j$ and $i^\alpha_{j+1}$, $v_0$ and $v_{w_\alpha+1}$ the switches incident to a boundary branch and $i_0$, $i_{w_\alpha}$ respectively. Let $b^\alpha_1$, $b^\alpha_{w_\alpha+1}$ be the boundary branches incident to $v^\alpha_1$, $v^\alpha_{w_\alpha}$, and let $s^\alpha_j$ be the buttress branch incident to $v^\alpha_j$.

\par Observe that there is exactly one cusp of a patch at the switch $v_j^\alpha$. Let $P^\alpha_j$ be the 
piece of this patch in $N(I_\alpha)$ containing the cusp at $v_j^\alpha$.
Also observe that there are two remaining components of patches intersecting $N(I_\alpha)$ which do not have cusps in $N(I_\alpha)$, and whose boundaries contain the boundary branches $b_1^\alpha$ and $b^\alpha_{w_\alpha + 1}$ respectively. Call these pieces $P_1^\alpha$ and $P_{w_\alpha+1}^\alpha$.
See \Cref{fig:NI sides}(a) for a summary of this notation.

\begin{figure}[h]
    \centering
    \subfigure[]{\includegraphics[width=0.55\textwidth]{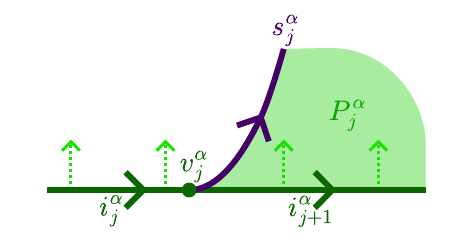}}
    \subfigure[]{\includegraphics[width=0.44\textwidth]{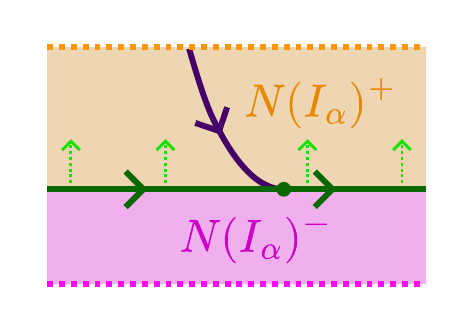}}
    \caption{(a) A local picture of the labels of structures in $N(I)$ and the orientations and coorientations. The green shaded region is the piece of the patch $P^\alpha_j$ which intersects $N(I_\alpha)$ and has a cusp at the switch $v^\alpha_j$. $i_j^\alpha$ and $i_{j+1}^\alpha$ are the infinitesimal branches incident to $v^\alpha_j$, and $s^\alpha_j$ is the buttress branch incident to $v^\alpha_j$. Arrows on branches indicate their orientation and the green dotted arrows indicate the coorientation. (b) The components $N(I_\alpha)^+$, $N(I_\alpha)^-$ of $N(I) \setminus (b_0^\alpha,i_1^\alpha,\ldots,i_{w_\alpha}^\alpha,b_{w_\alpha+1}^\alpha)$. $N(I_\alpha)^+$ is on the positively cooriented side of $I_\alpha$ and $N(I_\alpha)^-$ is on the negatively cooriented side. The buttress branch in (a) points away from $I_\alpha$ and the buttress branch in (b) points towards $I_\alpha$.}
    \label{fig:NI sides}
\end{figure}

Observe that $N(I_\alpha)$ is divided into two half-disks by the train path \linebreak
$(b_1^\alpha,2_1^\alpha,\ldots,i_{w_\alpha}^\alpha,b_{w_\alpha+1}^\alpha)$. We now fix local orientations and coorientations of $\tau$ in $N(I)$ so that we can distinguish these half-disks.
Since $\tau \cap N(I_\alpha)$ is contractible by \Cref{thm:infcomp}, we may fix a orientation on 
the piece of $\tau$ in $N(I_\alpha)$. 
By reversing the indices if necessary, we may assume that a vector pointing from $i_j$ towards $i_{j+1}$ is positively oriented. 
This orientation then determines a coorientation on $\tau$ in $N(I_\alpha)$ by the right hand rule.

Let $N(I_\alpha)^+$ be the half-disk component of $N(I_\alpha) \setminus (b_1^\alpha,i_2^\alpha,\ldots,i_{w_\alpha}^\alpha,b_{w_\alpha+1}^\alpha)$ on the positively cooriented side of $I_\alpha$ and $N(I_\alpha)^-$ the negatively cooriented one. See \Cref{fig:NI sides}(b).
\par If a positively oriented vector at the switch where a buttress branch meets $I_\alpha$ points into the large half-branch, then the orientation on that buttress branch points towards $I_\alpha$ and we say that the branch points \textit{\textbf{towards}} $I_\alpha$. Otherwise, we say the buttress branch points \textit{\textbf{away from}} $I_\alpha$. In \Cref{fig:NI sides} the buttress branch in (a) points towards $I_\alpha$ and the one in (b) points away from $I_\alpha$. If $s_j^\alpha$ intersects $N(I_\alpha)^+$, $P^\alpha_j$ is in $N(I_\alpha)^+$, otherwise it is in $N(I_\alpha)^-$.
\begin{lemma}\label{thm:infcompstruct}
    If $w_\alpha \ge 3$ the buttress branches intersecting $N(I_\alpha)^+$ either all point towards $I_\alpha$ or all point away from $I_\alpha$. The buttress branches intersecting $N(I_\alpha)^-$ all point the opposite way.
\end{lemma}
\begin{proof}
  Recall from part one of \Cref{wallprop} that when $w_\alpha \ge 3$ the infinitesimal and boundary branches of $I_\alpha$ are all dual to edges of $\Delta$ of the same color. For now, suppose they are all blue. 
  Also recall that branches of $\xi(\tau)$ are dual to edges of $\Delta$ and switches are dual to triangles, so adjacent branches correspond to edges of the same triangle. In particular, the consecutive infinitesimal branches $i_j^\alpha$, $i_{j+1}^\alpha$ correspond to edges of the triangle corresponding to the switch $v_j^\alpha$.
  The third edge in the boundary of this triangle must be colored red since the other two are blue, so by \Cref{wallprop} parts one and two it is dual to the buttress branch $s_j^\alpha$ incident to $v_j^\alpha$.

  If $P_j^\alpha$ intersects $N(I_\alpha)^+$, $s_j^\alpha$ must point towards $I_\alpha$, otherwise the red edge would be on the right side of the cusp at $v_j^\alpha$, contradicting the definition of the coloring. Similarly, if $P_j^\alpha$ intersects $N(I_\alpha)^-$, $s_j^\alpha$ must point away from $I_\alpha$. See \Cref{fig:inwardoutwardcolors}.

  \par By a symmetric argument, if the infinitesimal edges of $I_\alpha$ are red, the buttress branches intersecting $N(I_\alpha)^+$ must point away from $I_\alpha$ and those intersecting $N(I_\alpha)^-$ must point towards $I_\alpha$.
\end{proof}

\begin{figure}[h]
    \centering
    \includegraphics[width=0.66\textwidth]{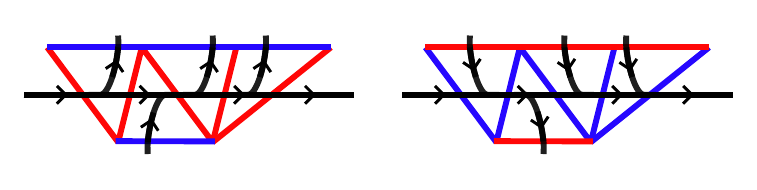}
    \caption{The coloring of edges of tetrahedra in a wall of width at least three determines whether buttress branches adjacent to infinitesimal edges point towards $I_\alpha$ or away from it.}
    \label{fig:inwardoutwardcolors}
\end{figure}
    
    Recall that the cusp of $P_j^\alpha$ at the infinitesimal switch $v_j^\alpha$ is bounded by two embedded train paths. We need one more fact about the structure of $N(I_\alpha)$ concerning the boundaries of $P^\alpha_j$:

    \begin{lemma}\label{thm:infcomppatch}
        For $2 \le j \le w_\alpha$ the train paths constituting the sides of $P_j^\alpha$ which meet at the infinitesimal cusp $v_j^\alpha$ both contain non-infinitesimal branches.
    \end{lemma}
    \begin{proof}
        One of the small branches incident to $v_j^\alpha$ is the buttress branch $s_j^\alpha$ which is contained in the train path on one side of $P^\alpha_j$, so in particular that train path contains a non-infinitesimal branch.
        \par On the other side of the cusp, observe that the only branches incident to $I_\alpha$ are boundary and buttress branches, which cannot be infinitesimal branches of other walls by \Cref{thm:infbranch}. See \Cref{fig:patchboundaryinwall}.
        \par When $w_\alpha \ge 3$, by \Cref{thm:infcompstruct}, the other buttress branches intersecting $N(I_\alpha)^{+/-}$ point the same way $s_j^\alpha$ does. 
        Hence, this side of $P_j^\alpha$ must exit $I_\alpha$, so it must contain a boundary or buttress branch.
        \par When $w_\alpha = 2$, $I_\alpha = (i_\alpha)$, so we need to argue that $i_\alpha$ itself does not itself constitute an entire side of $P^\alpha_j$. This is true because if $i_\alpha$ did constitute the side of a patch, the edge of $\Delta$ dual to $i_\alpha$ would be on the right side of a cusp of $\xi(\tau)$ on one side, and on the left side of a cusp of $\xi(\tau)$ on the other, so this edge would have to be colored both red and blue.
    \end{proof}

    \begin{figure}[h]
    \centering
    \includegraphics[width=0.66\textwidth]{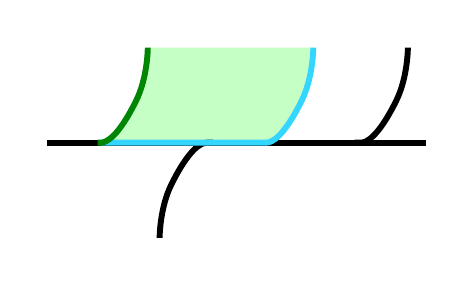}
    \caption{The boundary of a piece of a patch in $N(I_\alpha)$ consists of two train paths which must both contain non-infinitesimal edges.}
    \label{fig:patchboundaryinwall}
\end{figure}

\section{Constructing irreducible train tracks}\label{citts}

We are now ready to use $\tau$ to define a new $f$-invariant train track $\hat{\tau}$ which will have irreducible transition matrix as follows.

Without changing the embedding of $\tau$ outside of $N(I)$, delete each $I_\alpha$ and add one new switch $v_\alpha$ in each $N(I_\alpha)$. Re-attach to $v_\alpha$ the branches $b_1^\alpha$, $b_{w+1}^\alpha$, $s_2^\alpha \ldots s_{w_\alpha}^\alpha$ which were incident to $I_\alpha$ so that two of these branches $b$, $b'$ are adjacent in $\hat{\tau}$ if there is a patch $P$ of $S^\circ \setminus \tau$ so that $b$ and $b'$ bound the same piece of $P$ in $N(I_\alpha)$. 
Define the angle between $b$ and $b'$ to be zero if there is a cusp of $P$ in between $b$ and $b'$ in that piece of $P$ in $N(I_\alpha)$, and $\pi$ if not. 

\par To see that this defines a valid angle structure, we again consider two cases depending on whether or not $w_\alpha = 2$.
For $w_\alpha \ge 3$, observe that if there were more than one piece of a patch in $N(I_\alpha)^+$ not containing a cusp,
then there would be buttress branches in $N(I_\alpha)^+$ pointing both towards $I_\alpha$ and away from $I_\alpha$, contradicting \Cref{thm:infcompstruct}.
On the other hand, if there isn't a piece of a patch in $N(I_\alpha)^+$ containing a cusp, then there aren't any buttress branches in $N(I_\alpha)^+$, in which case the boundary branches are adjacent and assigned angle $0$.
Hence, when $w_\alpha \ge 3$ there is exactly one piece of a patch not containing a cusp in $N(I_\alpha)^+$. The same argument shows that there is exactly one piece of a patch not containing a cusp in $N(I_\alpha)^-$, so there are exactly two $\pi$ angles around $v_\alpha$. See \Cref{fig:tauhat}.

\par For $w_\alpha = 2$, $I_\alpha = (i_\alpha)$, so we may argue based on whether $i_\alpha$ is 
large or small at each endpoint, and whether the buttress branches at each endpoint are in $N(I_\alpha)^+$ or $N(I_\alpha)^-$. The only case in which the angle structure would not be well-defined is if $i_\alpha$ is small at both endpoints, and both buttress branches are on the same side of $N(I_\alpha)$, but this cannot happen because then $i_\alpha$ would constitute an entire side of a patch, contradicting \Cref{thm:infcomppatch}.

\begin{figure}[h]
    \centering
    \includegraphics[width=0.9\textwidth]{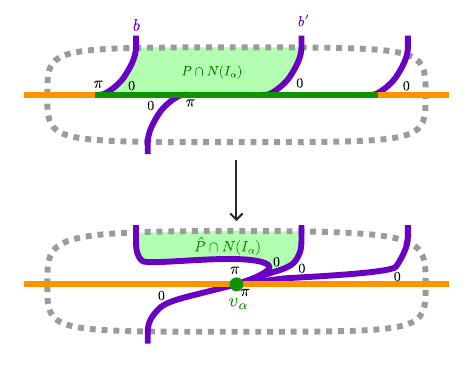}
    \caption{The construction of $\hat\tau$ from $\tau$ in $N(I_\alpha)$. $I_\alpha$ is replaced with a single switch $v_\alpha$ which is incident to all the branches which were incident to $I_\alpha$. Branches of $\hat{\tau}$ are adjacent if they are both in the boundary of the same piece of a patch. Adjacent branches are assigned angle $0$ if there is a cusp of the piece of the patch in between them in the boundary of the patch, and $\pi$ otherwise.}
    \label{fig:tauhat}
\end{figure}

\begin{lemma}\label{patchcorresp}
    The patches of $S^\circ \setminus \hat\tau$ correspond naturally to patches of $S^\circ \setminus \tau$, where corresponding patches coincide on the complement of $N(I)$.
\end{lemma}
\begin{proof}
    By definition, $\hat{\tau}$ agrees with $\tau$ outside $N(I)$. Let $P$ be a patch of $S^\circ \setminus \tau.$ Since $\tau \cap N(I)$ is contractible and $\partial P \subseteq \tau$ is not, 
$\partial P$ cannot be contained in $N(I)$, so $P \cap N(I)^c$ is non-empty. Then there is a patch $\hat{P}$ of $S \setminus \hat{\tau}$ which coincides with $P$ outside of $N(I)$, and has the same number of cusps in $N(I)$. Moreover, $P$ and $\hat{P}$ are smoothly isotopic if $\partial P \cap N(I)$ is embedded.
\end{proof}

\par There is also a correspondence between train paths in $\tau$ and $\hat{\tau}$:

\begin{lemma}\label{thm:trainpaths}
    Let $p\colon [0,1] \to \tau$ be a train path whose image is not contained in $I$. Then there is a train path $\hat{p}\colon [0,1] \to \hat{\tau}$ which runs over the corresponding branches in $p$, omitting the infinitesimal branches.
\end{lemma}
\begin{proof}
    For simplicity first suppose that $p([0,1]) \cap I$ has only one component, so it is contained in some $I_\alpha$. Let $p = (b_1 \ldots b_n)$.
    
    \par If $p([0,1]) \cap I$ consists of a single switch $p(t)$ for $t \in \{0,1\}$, either $b_1$ or $b_n$ is incident to $I_\alpha$ at $p(t)$, and there is a corresponding branch of $\hat{\tau}$ incident to $v_\alpha$. In this case, define $\hat{p}$ to be $p$ with the branch of $p$ incident to $I$ replaced with this branch incident to $v_\alpha$.
    
    \par If $p([0,1]) \cap I$ consists of a single switch $p(t)$ for $t \not \in \{0,1\}$, there is a $j$ such that $p(t)$ is the switch in between $b_j$ and $b_{j+1}$, where one of $b_j$, $b_{j+1}$ is a boundary branch and the other is a buttress branch. There are corresponding branches incident to $v_\alpha$ in $\hat{\tau}$. In this case, define $\hat{p}$ to be $p$ with $b_j$ and $b_{j+1}$ replaced with the corresponding branches in $\hat{\tau}$.
    
    \par Now suppose $p([0,1]) \cap I$ contains at least one branch. Since components of $I$ are embedded train paths by \Cref{thm:infcomp}, $p([0,1]) \cap I$ consists of a subpath $(b_j\ldots b_k)$ of infinitesimal branches contained in $I_\alpha$. If $j \ge 2$, $b_{j-1}$ is either an infinitesimal or boundary branch incident to $v_\alpha$, if $k \le n-1$, $b_{k+1}$ is either an infinitesimal or boundary branch incident to $v_\alpha$, and these branches are distinct.
    \par If $j=1$ then $k \le n-1$ since $p([0,1])$ is not contained in $I$, and $\hat{p} = (b_{k+1} \ldots b_n)$ is the train path in $\hat{\tau}$ with the desired properties. Similarly, if $k=n$ then $j \ge 2$ and we define $\hat{p} = (b_1 \ldots b_{j-1})$.
    \par Now suppose $j \ge 2$ and $k \le n-1$.  We consider three cases depending on whether $b_{j-1}$ and $b_{k+1}$ are buttress branches or boundary branches.
    \par If $b_{j-1}$ and $b_{k+1}$ are both buttress branches, they are on opposite sides of $N(I_\alpha)$. If $w_\alpha \ge 3$, by \Cref{thm:infcompstruct} one of $b_{j-1}$, $b_{k+1}$ points toward $I_\alpha$ and one points away, so $(b_{j-1},b_{k+1})$ is a train path in $\hat{\tau}$. If $w_\alpha = 2$, the infinitesimal branch $b_j=b_k$ must be small at both endpoints for there to be a train path through the buttress branches at both endpoints, so $(b_{j-1},b_{k+1})$ is again a train path in $\hat{\tau}$.
    
    \par If $b_{j-1}$ and $b_{k+1}$ are both boundary branches, there are patches $P_{j-1}$, $P_{k+1}$ so that $b_{j-1} \cap N(I_\alpha)$, $b_{k+1} \cap N(I_\alpha)$ are contained in distinct components of $P_{j-1} \cap N(I_\alpha)$, $P_{k+1} \cap N(I_\alpha)$ that do not contain cusps. Hence, in $\hat{\tau}$ the angle between $b_{j-1}$ and $b_{k+1}$ at $v_\alpha$ is $\pi$ on each side, so $(b_{j-1},b_{k+1})$ is again a train path in $\hat{\tau}$.
    \par If one of $b_{j-1}$, $b_{k+1}$ is a boundary branch and one is a buttress branch, again the angle between them at $v_\alpha$ is $\pi$ on each side.
    \par We conclude that in each case, $(b_{j-1},b_{k+1})$ is a train path in $\hat{\tau}$, so $\hat{p} = (b_1, \ldots, b_{j-1},b_{k+1},\ldots,b_n)$ is a train path in $\hat{\tau}$ from $p(0)$ to $p(1)$ running over the corresponding branches in $p$.
    \par Since components of $I$ are isolated by \Cref{thm:infcomp}, the general case follows.
\end{proof}

\par Now we construct a map $\hat{f}_{|\hat\tau}\colon \hat{\tau} \to \hat{\tau}$ which we will see can be extended to a train track map on $S^\circ$, with the extension being homotopic to $f$.
Recall that $f_{|\tau}: \tau \to \tau$ is the map obtained by restricting $\sigma \circ f$ to $\tau$.
Let $b$ be a branch of $\tau \setminus I$, $p_b$ the train path specified by restricting $f_{|\tau}$ to $b$, and $\hat{p}_b$ the corresponding train path on $\hat{\tau}$ from \Cref{thm:trainpaths}. Define $\hat{f}_{|\hat{\tau}}$ by setting $\hat{f}_{|\hat{\tau}}(b) = \hat{p}_b$.

\par Recall from \Cref{thm:patches} that $(\sigma \circ f)$ carries patches of $\tau$ to patches of $\tau$ with the same topology. Here we show that this structure is preserved by $\hat{f}_{|\hat{\tau}}$:

\begin{lemma}\label{thm:$n$-gonspreserved}
    
    Let $P$ be a patch of $S^\circ \setminus \tau$, $\hat{P}$ the corresponding patch of $S^\circ \setminus \hat{\tau}$, $Q$ the patch $P$ is carried to by $f$, and $\hat{Q}$ the patch of $S^\circ \setminus \hat{\tau}$ corresponding to $Q$. Then 
    $\partial \hat{Q} \subseteq \hat{f}_{|\hat{\tau}}(\partial \hat{P})$, and the interior of the region bounded by %$f(\hat{P})$ 
    $\hat{f}_{|\hat{\tau}}(\partial \hat{P})$ which contains the interior of $\hat{Q}$
    is exactly the interior of $\hat{Q}$. $\hat{P}$ and $\hat{Q}$ have the same number of cusps.
\end{lemma}
\begin{proof}
    Let $p_i$ be the sides of $P$, so $\partial P = \bigcup_i p_i$. Then $\partial \hat{P} \subseteq \bigcup_i \hat{p}_i$, where $\hat{p}_i$ are the train paths in $\hat{\tau}$ corresponding to $p_i$ from \Cref{thm:trainpaths}.
    Each $\hat{p}_i$ is non-empty by \Cref{thm:infcomppatch}, and by \Cref{thm:patches} we have that
    $$\bigcup_i (\sigma \circ f)(p_i) = \partial Q \cup \left ( \bigcup_i q_i \right ),$$ where $q_i$ are train paths 
    starting at the cusps of $\hat{Q}$ and going away from $\hat{Q}$.
    Let $\hat{q_i}$ be the train paths in $\hat{\tau}$ corresponding to $q_i$ and let $\hat{R}$ be the closed region bounded by $\hat{f}_{|\hat{\tau}}(\partial \hat{P})$ containing $\hat{Q}$. Then $\hat{R} \setminus \hat{Q} = \bigcup_i \hat{q_i}$, which has empty interior. $\hat{P}$ and $\hat{Q}$ have the same number of cusps as $P$ and $Q$ respectively by \Cref{patchcorresp}, and $P$ and $Q$ have the same number of cusps by \Cref{thm:patches}.
\end{proof}
\par Now we need to extend $\hat{f}_{|\hat{\tau}}$ to all of $S^\circ$.
Let $\{ \hat{P_i} \}$ be the patches of $S^\circ \setminus \hat{\tau}$
and let $P_i$ be the patch of $S^\circ \setminus \tau$ corresponding to $\hat{P}_i$.
Let $h\colon \tau \cap \overline{N(I)} \to \hat{\tau} \cap \overline{N(I)}$ be a map sending $(\tau \setminus I) \cap \overline{N(I)}$ to $(\hat{\tau} \setminus \{ v_\alpha\}) \cap N(I)$ homeomorphically and sending $I_\alpha$ to $v_\alpha$.
This forces $h$ to send switches to switches and branches of $\tau \setminus I$ to the corresponding branches of $\hat{\tau}$. It follows that for a switch $\hat{v}$ of $\hat{\tau} \setminus \{v_\alpha\}$, $\hat{f}_{|\hat{\tau}}(v) = (h \circ \sigma \circ f \circ h^{-1})(v)$,
and by isotoping $h$ 
we may assume that $\hat{f}_{|\hat{\tau}}(x) = (h \circ \sigma \circ f \circ h^{-1})(x)$ for every $x$ in $\hat{\tau} \setminus \{ v_\alpha\}$.
Also define $h$ on $\partial N(I)$ by the identity map, which is possible because $\partial N(I) \cap \tau$ coincides with $\partial N(I) \cap \hat{\tau}$.
\par Let $\hat{V}$ be a piece of $\hat{P_i} \cap N(I)$, and $V$ the corresponding piece of $P_i \cap N(I)$.
Topologically, $\hat{V}$ and $V$ are both disks 
and $h$ restricts to a map from $\partial V \to \partial \hat{V}$ which is a homeomorphism on $\partial V \setminus I$, and collapses $\partial V \cap I$ to a point. The map $h$ restricted to $\partial V$ is topologically conjugate to the quotient map $q_{|\partial D}: \partial D \to \partial D/C$, where $D$ is a disk and $C$ is an interval in $\partial D$. This map extends to a quotient map $q: D \to D/C$ which is a homeomorphism on the interior of $D$. Thus, $h$ can be extended to a map on $V$ which is a homeomorphism on the interior.
Doing this on each piece of each patch defines a map $h \colon \overline{N(I)} \to \overline{N(I)}$. which is the identity on $\partial N(I)$ and is a homeomorphism on $\overline{N(I)} \setminus \tau$.
 Define $\hat{f}\colon S^\circ \to S^\circ$ by
$$ \hat{f}(x) = \begin{cases}
    \hat{f}_{|\hat{\tau}}(x) & x \in \hat{\tau} \\
    (h_{|N(I) \setminus \hat{\tau}} \circ \sigma \circ f \circ h^{-1}_{|N(I) \setminus \hat{\tau}})(x) & x \in N(I) \setminus \hat{\tau} \\
    (\sigma \circ f)(x) & \textrm{otherwise}
\end{cases}$$

We claim $\hat{f}$ is continuous, which is clear away from $I$. To see that it is continuous on $N(I)$, let $U \subseteq N(I)$ be an open set in $N(I)$. Define $U_{\hat{\tau}} = U \cap \hat{\tau}$ and $U_c = U \setminus U_{\hat{\tau}}$. 
Observe that because $\hat{f}_{|\hat{\tau}}(x) = (h \circ \sigma \circ f \circ h^{-1})(x)$ for $x$ in $\hat{\tau} \setminus \{v_\alpha\}$, $\hat{f}_{|\hat{\tau}}^{-1}(U_{\hat{\tau}})$ coincides with $h(f^{-1}(\sigma^{-1}(h^{-1}(U_{\hat{\tau}}))))$
So $$\hat{f}^{-1}(U) = \hat{f}_{|\hat{\tau}}^{-1}(U_{\hat{\tau}}) \bigsqcup 
h_{|N(I) \setminus \hat{\tau}}(f^{-1}(\sigma^{-1}(h_{|N(I) \setminus \hat{\tau}}^{-1}(U_c))))
$$

$$= h(f^{-1}(\sigma^{-1}(h^{-1}(U_{\hat{\tau}})))) \bigsqcup h_{|N(I) \setminus \hat{\tau}}(f^{-1}(\sigma^{-1}(h_{|N(I) \setminus \hat{\tau}}^{-1}(U_c))))$$
$$= h(f^{-1}(\sigma^{-1}(h^{-1}(U))))
$$

which is open because $h$ is continuous.
It is also clear that $\hat{f}(\hat{\tau}) \subseteq \hat{\tau}$ and $\hat{f}$ preserves the tangent space of $\hat{\tau}$, so $\hat{f}
\colon S^\circ \to S^\circ$ is a train track map.

\par 
We now want to return to thinking of $f$ as a map on the full surface $S$ by filling the punctures $\{v_i\}$ of $S^\circ$ back in.
Since $\hat{f}: S^\circ \to S^\circ$ sends patches to patches bijectively and every patch contains exactly one puncture, it permutes the punctures of $S^\circ$, and we may extend $\hat{f}$ to the singular points of $S$ by continuity.

\begin{lemma}\label{htpic}
$\hat{f}$ and $f$ are homotopic as maps on $S$.
\end{lemma}
\begin{proof}
    Recall that $\sigma$ is homotopic the the identity, $N(I)$ is a disjoint union of disks, and $\hat{f}$ and $\sigma \circ f$ agree on the complement of $N(I)$.
    First we show that $\hat{f}$ and $f$ are homotopic as maps on $S^\circ$ rel punctures.
    Let $x_0$ be a point in $\tau \setminus N(I)$ and let $y_0 = f(x_0)$, so $x_0$ is also in $\hat{\tau}$, $y_0 = \hat{f}(x_0)$ and $y_0$ is also in both $\tau$ and $\hat{\tau}$. We will implicitly use $x_0$ and $y_0$ as basepoints for fundamental groups below, but suppress them from the notation.
    Since $\tau$ is a veering track, the patches of $S^\circ \setminus \tau$ are all once-punctured $n$-gons. By \Cref{patchcorresp}, the patches of $S^\circ \setminus \hat{\tau}$ are also all once-punctured $n$-gons. Hence, all of the patches of both tracks deformation retract onto their boundaries, which are contained in the tracks. Let $r\colon S^\circ \twoheadrightarrow \tau$, $\hat{r}\colon S^\circ \twoheadrightarrow \hat{\tau}$ be the retractions onto $\tau$ and $\hat{\tau}$ respectively, which both fix $x_0$ and $y_0$. Let 
    $r_*: \pi_1(S^\circ) \to \pi_1(\tau)$, 
    $\hat{r}_*: \pi_1(S^\circ) \to \pi_1(\hat{\tau})$
    be the induced homomorphisms, which are all isomorphisms.
    The inclusions $i\colon \tau \hookrightarrow S^\circ$ and $\hat{i}\colon \hat{\tau} \hookrightarrow S^\circ$ are homotopy inverses of $r$ and $\hat{r}$ and also fix $x_0$ and $y_0$. 
    Let $i_*: \pi_1(\tau) \to \pi_1(S^\circ)$,
    $\hat{i}_*: \pi_1(\hat\tau) \to \pi_1(S^\circ)$,
    be the induced isomorphisms.
     
    Let $(f_{|S^\circ})_*\colon \pi_1(S^\circ) \to \pi_1(S^\circ)$,
    $(f_{|\tau})_*\colon \pi_1(\tau) \to \pi_1(\tau)$,
    $(\hat{f}_{|\hat{\tau}})_*\colon \pi_1(\hat{\tau})\to\pi_1(\hat{\tau})$,
    $(\hat{f}_{|S^\circ})_*\colon \pi_1(S^\circ) \to \pi_1(S^\circ)$ be the homomorphisms induced by $f$ and $\hat{f}$. Those induced by $f$ are isomorphisms because $f$ is a homeomorphism. 
    
    Since $\hat{f}$ was defined to send patches of $S^\circ \setminus \hat{\tau}$ to the image under $f$ of the corresponding patches of $S^\circ \setminus \tau$, we have the following commutative diagram:

\begin{center}
    % https://tikzcd.yichuanshen.de/#N4Igdg9gJgpgziAXAbVABwnAlgFyxMJZABgBpiBdUkANwEMAbAVxiRAB120sB9ARgAUAZQB6nAMZYATuICUIAL6l0mXPkIoyfKrUYs2nbv2Fj2kmfKUrseAkT7kd9Zq0QcuvQZxx0ml5SAYNur2pNrUzvpuhp4C3r7+1mp2KABMjhF6ru5GXuwAFnQ4wPFMComBqrYayOnhui4GHsachcWl5YoBQck1AMwZDVE5saIS0nJdSdVEA-WR2THGY2YTljowUADm8ESgAGZSEAC2SGQgOBBI6SAMdABGMAwAClUhbgww+zggmY1uUh4ACopiBDickA4LldEAMhtlASCrGCjqdEDdLkgACzUO6PF5vFK3L4-P7DOIFIrAKQKAAE4xktKwsmBoPBaLhmMQAFYydkKW1qXSGeImSykQF2djqFyAGx8pqCrAKVnIqU8mUwgDsCuilOKytVktRZ01kN1IAE+x4wAAPitzOJykaDib0WbYRarTbbR1xWy3TjoUhefC2AKqfsVXbWlSOs6Ja6IYh5cHEDqw24I8UozaHRME4oKAogA
    \begin{tikzcd}
    \pi_1(S^\circ) \arrow[r, "r_*"] \arrow[d, "(f_{|S^\circ})_*"] & \pi_1(\tau) \arrow[r, "(\hat{r} \circ i)_*"] \arrow[d, "(f_{|\tau})_*"] & \pi_1(\hat{\tau}) \arrow[r, "\hat{i}_*"] \arrow[d, "(\hat{f}_{|\hat{\tau}})_*"] & \pi_1(S^\circ) \arrow[d, "(\hat{f}_{S^\circ})_*"] \\
    \pi_1(S^\circ) \arrow[r, "r_*"]                               & \pi_1(\tau) \arrow[r, "(\hat{r} \circ i)_*"]                            & \pi_1(\hat{\tau}) \arrow[r, "\hat{i}_*"]                                        & \pi_1(S^\circ)                                   
    \end{tikzcd}
\end{center}
The fundamental groups on the top row are based at $x_0$, the fundamental groups on bottom are based at $y_0$, and all of the maps are well defined with respect to those basepoints.
The left square is commutative because $f_{|\tau}$ is the restriction to $\tau$ of $(\sigma \circ f)$, and $\sigma$ is homotopic to the identity on $S^\circ$.
Similarly, the right square is commutative because $\hat{f}_{|\hat{\tau}}$ is the restrictions of $\hat{f}$.
The center square is commutative by the correspondence between train paths in $\tau$ and $\hat{\tau}$, which coincides with $\hat{r} \circ i$.
    Moreover, since the punctures of $S^\circ$ are all outside of $N(I)$ and $\hat{f}$ agrees with $f$ on the complement of $N(I)$, in particular $\hat{f}$ and $f$ agree on a system of cusp neighborhoods of the punctures, so they are homotopic rel punctures on $S^\circ$.
    \par Now $\hat{f}$ extends to the original surface $S$ by continuity, and $f$ was already defined there. In fact, since $\hat{f}$ and $f$ agree on a neighborhood of a singular point $x_i$, $f(x_i) = \hat{f}(x_i)$. We conclude that $\hat{f}$ and $f$ are homotopic.
\end{proof}

\begin{lemma}\label{tcarried}
    $f(\hat{\tau})$ is carried by $\hat{\tau}$ with support map $\hat{f} \circ f^{-1}$.
\end{lemma}
\begin{proof}
    Note that pseudo-Anosov maps are smooth away from the singularities of the unstable foliation, so $f$ is a diffeomorphism on $S^\circ$. Since $\hat{f}$ and $f$ are homotopic and $\hat{f}$, is a train track map, \Cref{supplemma} applies on $S^\circ$ with $t=\hat{f}$ and $g=f$, so $f(\hat{\tau}) \prec \hat{\tau}$ with support map $\hat{f} \circ f^{-1}$ on $S^\circ$. Then $\hat{f} \circ f^{-1}$ is also a support map for $f(\hat{\tau}) \prec \hat{\tau}$ on all of $S$.
\end{proof}

Let $\hat{M}$ be the transition matrix for $f(\hat{\tau}) \prec \hat{\tau}$ with support map $\hat{f} \circ f^{-1}$.

\begin{lemma}\label{firr}
    $\hat{M}$ is irreducible.
\end{lemma}
\begin{proof}
    The branches of $\hat{\tau}$ are in bijection with the branches of $\tau \setminus I$, and by definition the train path $\hat{f}_{|\tau}(b_1)$ passes over exactly the branches in the corresponding image of $f$. Therefore the carrying map for $\hat{f}$ is given by the submatrix of the adjacency matrix of $G$ corresponding to the edges of $\tau$ which are not infinitesimal. This subgraph is $G_\red$, which has irreducible adjacency matrix by \Cref{gred irr}.
\end{proof}

\par In short, we have used a veering train track $\tau$ to produce a new train track $\hat{\tau}$ whose branches are in correspondence with the non-infinitesimal branches of $\hat{\tau}$, and produced a support map $f(\hat{\tau}) \prec \hat{\tau}$ which has irreducible transition matrix. We summarize this as follows:
\begin{thm}
    $f(\hat{\tau})$ is carried by $\hat{\tau}$ with carrying map $\hat{f} \circ f^{-1}$. The transition matrix for $f(\hat{\tau}) \prec \hat{\tau}$ is irreducible.
\end{thm}

Putting all of this together, we prove \Cref{mainthm}:
\begin{proof}
    By \Cref{ttfs}, $f$ has an invariant train track $\tau$ which folds to its image, so we can build the veering triangulation $\Delta$ of the fully-punctured mapping torus using this folding sequence. If $\Delta$ does not contain a wall, the transition matrix for $f(\tau) \prec \tau$ is already irreducible. If not, we construct $\hat{\tau}$ as described in this section with transition matrix $\hat{M}$ for $f(\hat{\tau}) \prec_{(\hat{f} \circ f^{-1})} \hat{\tau}$.
    By \Cref{firr}, $\hat{M}$ is irreducible.
\end{proof}

\bibliography{bib}

@article{AT, title={Dynamics of veering triangulations: infinitesimal components of their flow graphs and applications}, volume={24}, ISSN={1472-2739}, url={https://msp.org/agt/2024/24-6/p10.xhtml}, DOI={10.2140/agt.2024.24.3401}, number={6}, journal={Algebraic \& Geometric Topology}, author={Agol, Ian and Tsang, Chi Cheuk}, year={2024}, month=oct, pages={3401–3453}, language={en} }

@article{PP,
author = {Athanase Papadopoulos and Robert C. Penner},
title = {{A characterization of pseudo-Anosov foliations.}},
volume = {130},
journal = {Pacific Journal of Mathematics},
number = {2},
publisher = {Pacific Journal of Mathematics, A Non-profit Corporation},
pages = {359 -- 377},
year = {1987},
}

@book{PH,
  title={Combinatorics of train tracks},
  author={Penner, Robert C and Harer, John L},
  number={125},
  year={1992},
  publisher={Princeton University Press}
}

@article{A,
  title={Ideal triangulations of pseudo-Anosov mapping tori},
  author={Agol, Ian},
  journal={Topology and geometry in dimension three},
  volume={560},
  pages={1--17},
  year={2011}
}

@book{FLP,
 ISBN = {9780691147352},
 author = {Albert Fathi and François Laudenbach and Valentin Poénaru and Djun M. Kim and Dan Margalit},
 publisher = {Princeton University Press},
 title = {Thurston's Work on Surfaces (MN-48)},
 urldate = {2024-04-29},
 volume = {48},
 year = {2012}
}

@article{LMT1, title={Flows, growth rates, and the veering polynomial}, volume={43}, ISSN={0143-3857, 1469-4417}, url={https://www.cambridge.org/core/journals/ergodic-theory-and-dynamical-systems/article/abs/flows-growth-rates-and-the-veering-polynomial/B79BE9FBDBE54CDE8C9D8A5285F4E7BF}, DOI={10.1017/etds.2022.63}, number={9}, journal={Ergodic Theory and Dynamical Systems}, author={Landry, Michael P. and Minsky, Yair N. and Taylor, Samuel J.}, year={2023}, month=sep, pages={3026–3107}, language={en} }

@article{P, title={Mutations and faces of the Thurston norm ball dynamically represented by multiple distinct flows}, volume={29}, ISSN={1364-0380}, url={https://msp.org/gt/2025/29-4/p07.xhtml}, DOI={10.2140/gt.2025.29.2105}, number={4}, journal={Geometry \& Topology}, author={Parlak, Anna}, year={2025}, month=jun, pages={2105–2173}, language={en} }

@article {L,
    AUTHOR = {Lackenby, Marc},
     TITLE = {Taut ideal triangulations of 3-manifolds},
   JOURNAL = {Geom. Topol.},
  FJOURNAL = {Geometry and Topology},
    VOLUME = {4},
      YEAR = {2000},
     PAGES = {369--395},
      ISSN = {1465-3060,1364-0380},
   MRCLASS = {57N10 (57M25 57Q15)},
  MRNUMBER = {1790190},
MRREVIEWER = {J.\ S.\ Birman},
       DOI = {10.2140/gt.2000.4.369},
       URL = {https://doi.org/10.2140/gt.2000.4.369},
}

@article {FR,
    AUTHOR = {Costantino, Francesco and Martelli, Bruno},
     TITLE = {An analytic family of representations for the mapping class
              group of punctured surfaces},
   JOURNAL = {Geom. Topol.},
  FJOURNAL = {Geometry \& Topology},
    VOLUME = {18},
      YEAR = {2014},
    NUMBER = {3},
     PAGES = {1485--1538},
      ISSN = {1465-3060,1364-0380},
   MRCLASS = {57R56 (22D12 57M27 57R52)},
  MRNUMBER = {3228457},
MRREVIEWER = {Torsten\ Asselmeyer-Maluga},
       DOI = {10.2140/gt.2014.18.1485},
       URL = {https://doi.org/10.2140/gt.2014.18.1485},
}

@article {EF,
    AUTHOR = {Lanneau, Erwan and Valdez, Ferr\'an},
     TITLE = {Computing the {T}eichm\"uller polynomial},
   JOURNAL = {J. Eur. Math. Soc. (JEMS)},
  FJOURNAL = {Journal of the European Mathematical Society (JEMS)},
    VOLUME = {19},
      YEAR = {2017},
    NUMBER = {12},
     PAGES = {3867--3910},
      ISSN = {1435-9855,1435-9863},
   MRCLASS = {37D40 (57M50)},
  MRNUMBER = {3730517},
MRREVIEWER = {Saadet\ \"Oyk\"u\ Yurtta\c s},
       DOI = {10.4171/JEMS/756},
       URL = {https://doi.org/10.4171/JEMS/756},
}

@article {F,
    AUTHOR = {Farre, James},
     TITLE = {Hamiltonian flows for pseudo-{A}nosov mapping classes},
   JOURNAL = {Comment. Math. Helv.},
  FJOURNAL = {Commentarii Mathematici Helvetici. A Journal of the Swiss
              Mathematical Society},
    VOLUME = {98},
      YEAR = {2023},
    NUMBER = {1},
     PAGES = {135--194},
      ISSN = {0010-2571,1420-8946},
   MRCLASS = {57K20 (30F60 32G15 37F34 57K43)},
  MRNUMBER = {4592854},
MRREVIEWER = {B\l a\.zej\ Szepietowski},
       DOI = {10.4171/cmh/551},
       URL = {https://doi.org/10.4171/cmh/551},
}

@article {J,
    AUTHOR = {Los, J\'er\^ome},
     TITLE = {Infinite sequence of fixed-point free pseudo-{A}nosov
              homeomorphisms},
   JOURNAL = {Ergodic Theory Dynam. Systems},
  FJOURNAL = {Ergodic Theory and Dynamical Systems},
    VOLUME = {30},
      YEAR = {2010},
    NUMBER = {6},
     PAGES = {1739--1755},
      ISSN = {0143-3857,1469-4417},
   MRCLASS = {37E30 (57M99)},
  MRNUMBER = {2736893},
MRREVIEWER = {Martin\ Scharlemann},
       DOI = {10.1017/S0143385709000832},
       URL = {https://doi.org/10.1017/S0143385709000832},
}

@article {Pe,
    AUTHOR = {Penner, R. C.},
     TITLE = {Bounds on least dilatations},
   JOURNAL = {Proc. Amer. Math. Soc.},
  FJOURNAL = {Proceedings of the American Mathematical Society},
    VOLUME = {113},
      YEAR = {1991},
    NUMBER = {2},
     PAGES = {443--450},
      ISSN = {0002-9939,1088-6826},
   MRCLASS = {57N05 (32G15 57M20 58F18)},
  MRNUMBER = {1068128},
MRREVIEWER = {Athanase\ Papadopoulos},
       DOI = {10.2307/2048530},
       URL = {https://doi.org/10.2307/2048530},
}

@article {M,
    AUTHOR = {McMullen, Curtis T.},
     TITLE = {Polynomial invariants for fibered 3-manifolds and
              {T}eichm\"uller geodesics for foliations},
   JOURNAL = {Ann. Sci. \'Ecole Norm. Sup. (4)},
  FJOURNAL = {Annales Scientifiques de l'\'Ecole Normale Sup\'erieure.
              Quatri\`eme S\'erie},
    VOLUME = {33},
      YEAR = {2000},
    NUMBER = {4},
     PAGES = {519--560},
      ISSN = {0012-9593},
   MRCLASS = {57M50 (30F60 32G15 37C85 37D40 57N10 57R30)},
  MRNUMBER = {1832823},
MRREVIEWER = {Athanase\ Papadopoulos},
       DOI = {10.1016/S0012-9593(00)00121-X},
       URL = {https://doi.org/10.1016/S0012-9593(00)00121-X},
}
\bibliographystyle{amsalpha} 
\end{document}